\def\ekv#1#2{\begeq\label{#1}#2\endeq}
\def\eekv#1#2#3{\begin{eqnarray}\label{#1}#2 \\ #3
\nonumber\end{eqnarray}}
\def\eeekv#1#2#3#4{\begin{eqnarray}\label{#1}#2 \\ #3
\nonumber\\#4\nonumber\end{eqnarray}}

\def\iint{\int\hskip -2mm\int} \def\iiint{\int\hskip -2mm\int\hskip
-2mm\int}
\def\iiiint{\int\hskip -2mm\int\hskip -2mm\int\hskip -2mm\int}
  
scaled \magstep1  
\font\liten=cmr10 at 8pt 
\font\lagom=cmr10 at 10pt
  \def\3{\vert \hskip -1pt\vert\hskip -1pt\vert }
\def\aby{arbitrary}

\def\ctf{canonical transformation}

\def\F{Fourier}

\def\hol{holomorphic}

\def\inv{^{-1}}

\def\no#1{(\ref{#1})}

\def\3{\vert\hskip -1pt\vert\hskip -1pt\vert }

 \def\pop{pseudodifferential operator}
   \def\rhs{right hand side}

  \def\trans{^t\hskip -2pt}

\def\Re{{\mathrm Re\,}} \def\Im{{\mathrm Im\,}}

\documentclass[12pt]{article} 
\usepackage{amssymb,latexsym} \title{
Pseudodifferential operators and weighted normed symbol spaces} \author{J.
Sj{\"o}strand\\ \lagom CMLS\\ \lagom Ecole Polytechnique\\ \lagom FR  91120
Palaiseau c\'edex, France\\ \lagom johannes@{}math.polytechnique.fr \\ \lagom 
UMR7640--CNRS}  \date{}
\def\wrtext#1{\relax\ifmmode{\leavevmode\hbox{#1}}\else{#1}\fi}
 \def\begeq{\begin{equation}}
\def\endeq{\end{equation}}

\def\part#1{\frac{\partial}{\partial #1}}

\renewcommand{\exp}{\mbox{\rm exp\,}}

\newtheorem{dref}{Definition}[section]
\newtheorem{lemma}[dref]{Lemma}
\newtheorem{theo}[dref]{Theorem}
\newtheorem{prop}[dref]{Proposition}

\newtheorem{ex}[dref]{Example}
\newtheorem{cor}[dref]{Corollary}
\newtheorem{ass}[dref]{Assumption}

\newenvironment{proof}{\par\noindent{{\bf Proof }}}{\hfill$\Box$
\medskip}

\begin{document}
\maketitle

\begin{abstract} \smallskip
\par In this work we study some general classes of pseudodifferential operators
 where the classes of symbols are defined in terms of phase space 
estimates.
 \medskip\par
\centerline{\bf R\'esum\'e}
\smallskip
\par  On \'etudie des classes g\'en\'erales d'op\'erateurs 
pseudodiff\'erentiels
dont les classes de symboles sont d\'efinis en termes d'\'estimations dans 
l'espace de phase.
\end{abstract}

\tableofcontents
\section{Introduction}\label{int}
\setcounter{equation}{0}

\par This paper is devoted to 
pseudodifferential operators with symbols of limited regularity. The
author
\cite{Sj1} introduced the space of 
symbols $a(x)$ on the phase space $E={\bf R}^n\times ({\bf R}^n)^*$
with the  property that
\ekv{int.1}
{\vert\widehat{\chi_\gamma a}(x^*)\vert \le  F(x^*),\ \forall \gamma
\in \Gamma }  for some $L^1$  function $F$ on $E^*$. Here the hat
indicates that we take the \F{} transform,  $\Gamma \subset E$ is a
lattice and $\chi_\gamma (x)=\chi_0(x-\gamma )$ form a partition of
unity, $1=\sum_{\gamma \in\Gamma }\chi_\gamma $, $\chi_0 \in  {\cal
S}(E)$. A.~Boulkhemair \cite{Bo1} noticed that this space is identical
to a space that  he had defined differently in \cite{Bo0}.

It was shown among other things that this space of symbols is an
algebra  for the ordinary multiplication and that this fact persists
after  quantization, namely the corresponding pseudodifferential
operators  (say under Weyl quantization) form a non-commutative
algebra:  If $a_1,a_2$ belong to the class above with corresponding
$L^1$ functions  $F_1$ and $F_2$ then $a_1^w\circ a_2^w=a_3^w$ where
$a_3$ belongs to the same  class and as a correponding function we may
take $F_3=C_NF_1*F_2*
\langle \cdot  \rangle^{-N}$ for any $N> 2n$. Here $*$ indicates 
convolution and $a^w:{\cal S}({\bf R}^n)\to {\cal S}'({\bf R}^n)$ is
the  Weyl quantization of the symbol $a$, given by
\ekv{int.2}
{ a^wu(x)={1\over (2\pi )^n}\iint e^{i(x-y)\cdot \theta }a({x+y\over
2},\theta )u(y)dyd\theta .  } The definition \no{int.1} is
independent of the choice of lattice and the corresponding function
$\chi_0$. When passing to a different choice, we may have to change 
the function $F$ to
$m(x^*)=F*\langle \cdot \rangle^{-N_0}$ for any fixed  $N_0>2n$. We
then gain the fact that the weight $m$ is an order function in the
sense that
\ekv{int.3}
{ m(x^*)\le C_0\langle x^*-y^* \rangle ^{N_0}m(y^*),\ x^*,y^*\in E^*.
} (See \cite{DiSj} where this notion is used for developing a fairly
simple  calculus of semi-classical pseudodifferential operators,
basically a special  case of H\"ormander's Weyl calculus \cite{Ho}.)

The space of functions in \no{int.1} is a special case of  the
modulation spaces of  H.G.~Feichtinger (see \cite{Fe, FeGr}), and the
relations  between these spaces and pseudodifferential operators  have
been developed by many authors; K.~Gr\"ochenig \cite{Gr2,
Gr3}, Gr\"ochenig,  T. Strohmer \cite{GrSt}, K.~Tachigawa \cite{Ta},
J.~Toft \cite{To1}, A. Holst, J. Toft, P. Wahlberg \cite{HoToWa}. 
Here we could  mention that
Boulkhemair \cite{Bo2} proved $L^2$-continuity for Fourier integral
operators with symbols and phases in the original spaces of the type
\no{int.1}, that T. Strohmer \cite{St} has applied the theory to
problems in mobile communications and that  Y.~Morimoto and N.~Lerner
\cite{MoLe} have used the original space to prove a version of the
Fefferman-Phong inequality for pseudodifferential operators with
symbols of low regularity. This result was recently improved by 
Boulkhemair \cite{Bo5}.

Closely related works on pseudodifferential - and Fourier - integral
operators with symbols  of limited regularity include the works 
of  Boulkhemair
\cite{Bo3, Bo4}, and many others also contain a study of  when such
operators or related Gabor localization operators  belong to to
Schatten-von Neumann classes: E.~Cordero, Gr\"ochenig \cite{CoGr1,
CoGr2}, C.~Heil, J.~Ramanathan, P.~Topiwala
\cite{HeRaTo}, Heil \cite{He}, J.~Toft \cite{To2}, and M.W.~Wong \cite{Wo}.

The present work has been stimulated by these developments and the
prospect of using  ``modulation type weights'' to get more flexibility
in the calculus of pseudodifferential  operators with limited
regularity. In the back of our head there were also some very
stimulating discussions with J.M.~Bony and N.~Lerner from the time  of
the writing of \cite{Sj1, Sj2} and at that time Bony explained to the 
author a
nice very general  point of view of A. Unterberger \cite{Un} 
for a direct microlocal analysis of very general
classes of operators. Bony used it in his
work \cite{Bon1} and showed how his approach could be 
applied to recover and generalize the space 
in \cite{Sj1}. However, the aim of the work \cite{Bon1} was to develop
a very general theory of  Fourier integral operators related to
symplectic metrics of H\"ormander's Weyl calculus of
pseudodifferential operators, and the relation with \cite{Sj1} was
explained very briefly.  See \cite{Bon2} for even more
general classes of Fourier integral operators.

\par
In the present paper we make a direct generalization of the spaces of \cite{Sj1}.
Instead of using order functions only depending on $x^*$ we can now allow arbitrary order 
functions $m(x,x^*)$. See Definition \ref{sy1} below. In Proposition 
\ref{sy4} we show that this definition gives back the spaces above 
when the weight $m(x^*)$ is an order function of $x^*$ only.

\par In Section \ref{l2} we consider the quantization of our symbols 
and show how to define
an associated effective kernel on $E\times E$, $E=T^*{\bf R}^n$, which is 
${\cal O}(1)m(\gamma (x,y))$ where $\gamma (x,y)=({x+y\over 2},J\inv (y-x))$ 
and $J:E^*\to E$ is the natural Hamilton map induced by the symplectic 
structure. 
We show that if the effective kernel is the kernel of a bounded operator
: $L^2(E)\to L^2(E)$ then our pseudodifferential operator is bounded in 
$L^2({\bf R}^n)$. In particular if $m=m(x^*)$ only depends on $x^*$, we recover 
the $L^2$-boundedness when $m$ is integrable. This result was obtained previously by 
Bony \cite{Bon1}, but our approach is rather different.

\par In Section \ref{co} we study the composition of pseudodifferential 
operators in our classes. If $a_j$ are symbols associated 
to the order functions $m_j$, $j=1,2$, then the Weyl composition is a 
well defined symbol associated to the order function $m_3(z,z^*)$ given 
in \no{co.12}, provided that the integral there converges for at least 
one value of $(z,z^*)$ (and then automatically for all other values by 
Proposition \ref{co1}). This statement is equivalent to the corresponding 
natural one for the effective kernels, namely the composition is well 
defined if the composition of the majorant kernels 
$m_1({x+y\over 2},J\inv (y-x))$ 
and $m_2({x+y\over 2},J\inv (y-x))$  is well-defined, see \no{co.16}, 
\no{co.17}.

In Section \ref{di} we simplify the results further (for those readers 
who are familiar with Bargmann transforms from the FBI - complex Fourier 
integral operator point of view). 

\par In Section \ref{cp} we use the same point 
of view to give a simple sufficient condition on the order function $m$ 
and the index $p\in [1,\infty ]$, for  the quantization $a^w$ to
belong to the Schatten--von Neumann class $C_p$
for every symbol 
$a$ belonging to the symbol class with weight $m$. See \cite{To2, To3, HoToWa,
GrHe, GrHe2} for related results and ideas.

\par In Section \ref{ge} we finally generalize our results by replacing 
the underlying space $\ell^\infty $ on certain lattices by  more general 
translation invariant Banach spaces. We believe that this generalization 
allows to 
include modulation spaces, but we have contented ourselves by establishing 
results allowing to go from properties on the level of lattices to the 
level of 
pseudodifferential operators. The results could undoubtedly be even further 
generalized. In this section and the preceding one, we have been inspired by 
the use of lattices and amalgan spaces in time frequency analysis, 
in particular by the work of Gr\"ochenig and Strohmer \cite{GrSt} that 
uses previous results by Fournier--Stewart \cite{FoSt} and Feichtinger 
\cite{Fe2}.

\par We have chosen to work with the Weyl quantization, but it is clear 
that the results
carry over with the obvious modifications to other quantizations like the 
Kohn-Nirenberg one, actually for the general symbol-spaces under consideration 
the results could also have been formulatated directly for classes of integral 
operators. 

\par Similar ideas and results have been obtained in many other 
works, 
out of 
which some are cited above and later in the text.

\medskip

\par\noindent\bf Acknowledgements. \rm We thank J.M. Bony for a very 
stimulating 
and helpful recent discussion. The author also thanks K. Gr\"ochenig,
T. Strohmer, A. Boulkhemair and J. Toft for several helpful 
comments and references.

\section{Symbol spaces}\label{sy}
\setcounter{equation}{0}

\par Let $E$ be a $d$-dimensional real vector space. We say that 
$m:E\to ]0,\infty [ $ is an order function on $E$ if there exist 
constants $C_0>0$, $N_0\ge 1$, such that 
\ekv{sy.1}
{
m(\rho )\le C_0\langle \rho -\mu \rangle ^{N_0}m(\mu ),\ \forall \rho ,\mu 
\in E.
}
Here $\langle \rho -\mu \rangle =(1+\vert \rho -\mu \vert ^2)^{1/2}$ 
and $\vert\,\vert$ is a norm on $E$.

Let $E$ be as above, let $E^*$ be the dual space and let $\Gamma $ be a 
lattice in  $E\times E^*$, so that $\Gamma ={\bf Z}e_1+{\bf Z}e_2+...+{\bf Z}
e_{2d}$ where $e_1,...,e_{2d}$ is a basis in $E\times E^*$. 
Let $\chi \in{\cal S}(E\times E^*)$ have the property that 
\ekv{sy.2}
{\sum_{\gamma \in\Gamma }\tau _\gamma \chi =1,\quad \tau_\gamma 
 \chi (\rho )= \chi (\rho -\gamma ).}
Let $m$ be an order function on $E\times E^*$, $a\in {\cal S}'(E)$.

\begin{dref}\label{sy1}\rm We say that $a\in\widetilde{S}(m)$ 
if there is a constant $C>0$ such that 
\ekv{sy.3}
{
\Vert \chi _\gamma ^wa\Vert\le Cm(\gamma ),\ \gamma \in \Gamma ,
} 
where $\chi _\gamma =\tau_\gamma \chi $ and $\chi _\gamma^w $ denotes the Weyl quantization of $\chi _\gamma $. The norm will always be the the one in 
$L^2$ if nothing else is indicated.
\end
{dref}

To define the $L^2$-norm we need to choose a Lebesgue measure on $E$, but 
clearly that can only affect the choice of the constant in \no{sy.3}. 

\begin{prop}\label{sy2} $\widetilde{S}(m)$ is a Banach space with 
$\Vert a\Vert_{\widetilde{S}(m)}$ equal to the smallest possible 
constant in 
\no{sy.3}. Changing $\Gamma$, $\chi$ and replacing the $L^2$ norm by 
the $L^p$-norm 
for any $p\in [1,\infty ]$ in the above definition, gives rise to the same 
space with an equivalent norm.
\end{prop}
\begin{proof} The Banach space property will follow from the other arguments 
so we do not treat it explicitly. Let $m,\Gamma ,a$ be as in Definition 
\ref{sy1}.

\par Let $\widetilde{\Gamma}$ be another lattice and let $\widetilde{\chi }$ 
be another function with the same properties as $\chi$. We have to show that 
$$\Vert \widetilde{\chi}^w_{\widetilde{\gamma }}a\Vert_{L^p}\le  
\widetilde{C}m(\widetilde{\gamma }),\quad \widetilde{\gamma }\in 
\widetilde{\Gamma }.$$
\begin{lemma}\label{sy3}
$\exists \psi \in {\cal S}(E\times E^*)$ such that $\sum_{\gamma \in\Gamma }\psi _\gamma^w \chi _{\gamma }^w=1$, where $\psi _\gamma =\tau_\gamma \psi $.
\end{lemma}
\begin{proof}
Let $\widetilde{\chi}\in{\cal S}(E\times E^*)$ be equal to 1 near $(0,0)$, 
and put $\widetilde{\chi }^\epsilon (x,\xi )=\widetilde{\chi}
(\epsilon (x,\xi ))$. Then $\sum_{\gamma \in\Gamma }(1-\widetilde{\chi }_\gamma ^\epsilon )\# \chi _\gamma \to 0$ in $S^0(E\times E^*)$, when $\epsilon \to 0$, so for $\epsilon >0$ small enough, 
$$\sum_{\gamma \in\Gamma }(\widetilde{\chi }_\gamma ^\epsilon )^w\chi _\gamma 
^w=1-\sum_{\gamma \in\Gamma }(1-\widetilde{\chi }_\gamma ^\epsilon )^w\chi _\gamma 
^w$$ 
has a bounded inverse in ${\cal L}(L^2,L^2)$. Here $S^0$ is the space of 
all $a\in C^\infty (E\times E^*)$ that are bounded with all their 
derivatives. By a version of the Beals lemma 
(see for instance \cite{DiSj}), we then know that the inverse is of the form 
$\Psi ^w$ where $\Psi \in S^0$.  
Also $\tau_\gamma \Psi =\Psi $, $\gamma \in\Gamma $. Put $\psi_\gamma ^w =
\Psi ^w\circ (\widetilde{\chi }_\gamma ^\epsilon )^w$ for $\epsilon $ small 
enough and fixed, so that $\psi _\gamma =\tau_\gamma \psi _0$, 
$\psi _0\in{\cal S}$ (using for instance the simple pseudodifferential 
calculus in \cite{DiSj}). Then $\sum_{\gamma }\psi _\gamma ^w
\chi _\gamma ^w=1$.
\end{proof}

\par Now, write 
$$\widetilde{\chi }_{\widetilde{\gamma 
}}^wa=\sum_{\gamma \in \Gamma }\widetilde{\chi }_{\widetilde{\gamma }}^w \psi 
_\gamma^w\chi _\gamma ^w a.$$
Here (using for instance \cite{DiSj})
$$\Vert \widetilde{\chi }_{\widetilde{\gamma }}\psi _\gamma ^w\Vert_{{\cal L}(L^2,L^p)}\le C_{p,N}\langle \widetilde{\gamma }-\gamma \rangle ^{-N},\ 1\le p\le \infty ,\,N\ge 0.$$
Hence, if $N$ is large enough, \begin{eqnarray}\label{sy.3.5}
\Vert \widetilde{\chi }_{\widetilde{\gamma }}^wa\Vert_{L^p}&\le&
C_{p,N}\sum_{\gamma \in\Gamma }\langle \widetilde{\gamma }-\gamma 
\rangle ^{-N}\Vert \chi _{\gamma }^w a\Vert_{L^2}\\
&\le& \widetilde{C}_{p,N,a}\sum_{\gamma \in \Gamma }\langle 
\widetilde{\gamma }-\gamma \rangle ^{-N}m(
\gamma )\nonumber
\\
&\le &\widehat{C}_{p,N,a,m}(\sum_{\gamma \in \Gamma }\langle 
\widetilde{\gamma }-\gamma \rangle ^{-N+N_0})m(\widetilde{\gamma })
\nonumber
\\
&\le &\check{C}m(\widetilde{\gamma }).\nonumber
\end{eqnarray} 

\par Conversely, if $\Vert \widetilde{\chi }_{\widetilde{\gamma }}^wa
\Vert_{L^p}\le {\rm Const\,}m(\widetilde{\gamma })$, $\widetilde{\gamma }\in 
\widetilde{\Gamma }$, we see that by the same arguments that $\Vert 
\chi _\gamma ^wa\Vert_{L^2}\le {\cal O}(1)m(\gamma )$, $\gamma \in\Gamma$. 
\end{proof}

Next, we check that this is essentially a generalization of a space introduced 
by Sj\"ostrand \cite{Sj1} and independently and in a different way by 
Boukhemair \cite{Bo0}. It is a 
special case of more general modulation spaces (see \cite{Fe,FeGr}).
That follows from the next result if we take an order function $m(x,x^*)$ 
independent of 
$x$.
\begin{prop}\label{sy4}
Let $m=m(x,x^*)$ be an order function on $E\times E^*$
and let $\chi \in {\cal S}(E)$, $\sum_{j\in J}\chi _j=1$, 
where $J\subset E$ is a lattice and 
$\chi _j=\tau_j\chi$. Then 
\ekv{sy.4}
{\widetilde{S}(m)=\{ a\in {\cal S}'(E);\, \exists C>0,\, \vert 
\widehat{\chi _ju}(x^*)\vert \le Cm(j,x^* )\}.
} 
\end{prop}

\begin{proof}
Let $K\subset E^*$ be a lattice and choose $\chi ^*\in{\cal S}(E^*)$, 
such that 
$\sum_{k\in K}\chi_k^*=1$, where $\chi _k=\tau_k\chi $. If $a$ belongs to 
the set in the \rhs{} of \no{sy.4}, then by Parseval's relation,
\ekv{sy.5}
{
\Vert \chi _k^*(D)(\chi _j(x)u(x))\Vert _{L^2}\le \widetilde{C}m(j,k).
}
Now $\chi _k^*(D)\circ \chi _j(x)=\chi _{j,k}^w$, where 
$\chi_{j,k}=\tau_{j,k}\chi _{0,0}$, 
$\chi _{0,0}\in {\cal S}$, $(j,k)\in J\times J^*$, 
so $a\in \widetilde{S}(m)$. Conversely, if $a\in \widetilde{S}(m)$, 
we get \no{sy.5}. According to Proposition \ref{sy2}, we can replace 
the $L^2$ norm by any $L^p$ norm, and the proof shows that we can 
equally well replace the $L^2$ norm that of 
${\cal F}L^p$. Taking ${\cal F}L^\infty $, we get
$$\Vert \chi _k^*(x^* )\widehat{\chi _ju}(x^* )\Vert_{L^\infty }\le 
\widehat{C}m(j,k),$$
and since $m$ is an order function, we deduce that $a$ belongs to 
the set in the right hand side of \no{sy.4}. 
\end{proof}


\section{Effective kernels and $L^2$-boundedness}\label{l2}
\setcounter{equation}{0}
A closely related notion for effective kernels in terms of short time Fourier
transforms has been introduced by Gr\"ochenig and Heil \cite{GrHe}.
\par We now take $E={\bf R}^{2n}\simeq T^*{\bf R}^n$. If $a,b\in {\cal S}(E)$, we let 
\ekv{l2.1}
{
a\# b=(e^{{i\over 2}\sigma (D_x,D_y)}a(x)b(y))_{y=x}
}
denote the Weyl composition so that $(a\# b)^w=a^w\circ b^w$. Here 
$\sigma (D_{x,\xi },D_{y,\eta })=D_\xi\cdot D_y-D_x\cdot D_\eta $ where 
we write $(x,\xi)$, $(y,\eta )$ instead of $x$, $y$ whenever convenient.

\par We know that the Weyl composition is still well-defined when $a,b$ belong to 
various symbol spaces like
\ekv{l2.2}
{
S(m)=\{ a\in C^{\infty }(E);\, \vert D_x^\alpha a(x) \vert\le C_\alpha m(x)\} ,
} 
when $m$ is an order function on $E$. (See Example \ref{co3} below for a 
straight forward generalization.)

Let $\ell (x)=x\cdot x^*$ be a linear form on $E$ and let $a$ be a symbol. Then,
\eeekv{l2.3}
{
e^{i\ell}\# a&=& e^{{i\over 2}\sigma (D_x,D_y)}(e^{i\ell (x)}a(y))_{y=x}
}
{
&=& (e^{i\ell (x)}e^{{i\over 2}\sigma (\ell '(x),D_y)}a(y))_{y=x}
}
{&=& e^{i\ell (x)}(e^{{1\over 2}H_\ell}a)
}
where $H_\ell =\ell '_\xi \cdot {\partial \over \partial x}-\ell '_x \cdot 
{\partial \over \partial \xi}$ (with ``$x=(x,\xi )$'') is the Hamilton field of $\ell$. Similarly,
\ekv{l2.4}
{a\# e^{i\ell}=e^{i\ell (x)}(e^{-{1\over 2}H_\ell}a).
}

\par From \no{l2.3}, \no{l2.4}, we get 
\ekv{l2.5}
{e^{i\ell}\# a\# e^{-i\ell}=e^{H_\ell}a,
}
where we notice that $(e^{H_\ell}a)(x)=a(x+H_\ell )$, and 
\ekv{l2.6}
{e^{i{m\over 2}}\# a\# e^{i{m\over 2}}=e^{im}a,}
if $m$ is a second linear form on $E$.

\par If $a\in {\cal S}(E)$ is fixed, we may consider that $a$ is concentrated near 
$(0,0)\in E\times E^*$. Then we say that $e^{-H_\ell}e^{im}a$ is concentrated 
near $(H_\ell ,m)\in E\times E^*$. Conversely, if $b$ is concentrated near a 
point $(x_0,x_0^*)\in E\times E^*$, we let $y_0^*\in E^*$ be the unique 
vector with $x_0=H_{y_0^*}$ and write
\ekv{l2.7}
{
b=e^{-H_{y_0^*}}e^{ix_0^*}a= e^{-iy_0^*}\# e^{i{x_0\over 2}}\# a
\# e^{i{x_0\over 2}}\# e^{iy_0^*},
}
where $a$ is concentrated near $(0,0)\in E\times E^*$. 

\par To make this more precise, let (as in \cite{Sj3})
\ekv{l2.8}
{
Tu=C\int e^{i\phi (x,y)}u(y)dy,\quad C>0,
}
be a generalized Bargmann transform where $\phi (x,y)$ is a quadratic form on 
${\bf C}^n\times{\bf C}^n$ with $\det \phi ''_{xy}\ne 0$, $\Im \phi ''_{yy}>0$, 
and with $C>0$ suitably chosen, so 
that $T$ is unitary $L^2({\bf R}^n)\to H_{\Phi }({\bf C}^n)={\rm Hol\,}
({\bf C}^n)\cap L^2(e^{-2\Phi (x)}L(dx))$, where $L(dx)$ denotes the Lebesgue 
measure on ${\bf C}^n$ and $\Phi $ is the strictly plurisubharmonic quadratic form given by 
\ekv{l2.9}
{
\Phi (x)=\sup_{y\in{\bf R}^n} -\Im \phi (x,y).
}
We know (\cite{Sj3}) that if $\Lambda _{\Phi }=\{ (x,{2\over i}{\partial \Phi 
\over \partial x});\, x\in{\bf C}^n\}$, then 
\ekv{l2.10}
{
\Lambda _{\Phi }=\kappa _{T}(E),
}
where 
\ekv{l2.11}
{
\kappa _T:{\bf C}^{2n}\simeq E^{\bf C}\ni (y,-\phi '_y(x,y))\to 
(x,\phi '_x(x,y))\in{\bf C}^{2n}
}
is the linear \ctf{} associated to $T$. Here ${\partial \over \partial x}=
{1\over 2}({\partial \over \partial \Re x}+{1\over i}
{\partial \over \partial \Im x})$, following standard conventions in 
complex analysis.

\par If $a\in S^0(E)$ we have an exact version of Egorov's theorem, saying 
that 
\ekv{l2.12}
{
Ta^wT\inv =\widetilde{a}^w,
}
where $\widetilde{a}\in S^0(\Lambda _{\Phi })$ is given by 
$\widetilde{a}\circ \kappa _T=a$. In \cite{Sj3} it is dicussed how to define 
and estimate the Weyl quantization of symbols on the Bargmann transform side, 
by means of almost \hol{} extensions and contour deformations. We retain from 
the proof of Proposition 1.2 in that paper that 
\ekv{l2.13}
{
\widetilde{a}^wu(x)=\int e^{\Phi (x)}K^{\rm eff}_{\widetilde{a}}(x,y) u(y)e^{-\Phi (y)}
L(dy),\ u\in H_\Phi ({\bf C}^n),
}
where the kernel is non-unique but can be chosen to satisfy
\ekv{l2.14}
{
K^{\rm eff}_{\widetilde{a}}(x,y)={\cal O}_N(1)\langle x-y\rangle ^{-N},
}
for every $N\ge 0$. (This immediately implies the Calder\'on-Vaillancourt 
theorem for the class ${\rm Op\,}(S^0(E))$.)

\par If $a\in {\cal S}(E)$, then for every $N\in {\bf N}$
\ekv{l2.15}
{
\vert K^{\rm eff}_{Ta^wT\inv}(x,y) \vert \le C_N(a)\langle x \rangle^{-N}\langle y \rangle
^{-N},\ x,y\in{\bf C}^n,
}
where $C_N(a)$ are seminorms in ${\cal S}$.

\par Identifying $x\in{\bf C}^n$ with $\kappa _T\inv (x,{2\over i}{\partial 
\Phi \over \partial x})\in E$, we can view $K^{\rm eff}_{Ta^wT\inv}$ as a function 
$K^{\rm eff}_{a^w}(x,y)$ on $E\times E$ and \no{l2.15} becomes
\ekv{l2.16}
{
\vert K^{\rm eff}_{a^w}(x,y) \vert \le C_N(a)\langle x \rangle^{-N}\langle y \rangle^{-N}, \ x,y\in E.
}

\par Now, let $b$ in \no{l2.7} be concentrated near 
$(x_0,x_0^*)=(J{y_0^*},x_0^*)\in E\times E^*$ with $a\in{\cal S}(E)$, where we let 
$J:E^*\to E$ be the map $y^*\mapsto H_{y^*}$ (and we shall prefer to write 
$Jy^*$ when we do not think of this quantity as a constant coefficient vector 
field). Then by \no{l2.5}--\no{l2.7}, we have 
\ekv{l2.17}
{
b=e^{-iy_0^*}\# e^{ix_0^*/2}\# a\# e^{ix_0^*/2}\#e^{iy_0^*},
}
\ekv{l2.18}
{
b^w=e^{-i(y_0^*)^w}\circ e^{i(x_0^*)^w/2}\circ a^w\circ e^{i(x_0^*)^w/2}
\circ e^{i(y_0^*)^w}.
}
\par
Now it is wellknown that if $z^*\in E^*$ then 
$e^{-i(z^*)^w}=(e^{-iz^*})^w$ is a unitary operator that can be viewed 
as a quantization of the phase space 
translation $E\ni x\mapsto x+H_{z^*}\in E$. On the Bargmann transform side 
these quantizations can be explicitly represented as magnetic translations, 
i.e. translations made unitary by multiplication by certain weights. In fact, 
let $\ell (x,\xi )=x_0^*\cdot x+x_0\cdot \xi $ be a linear form on 
${\bf C}^{2n}$ which is real on $\Lambda _\Phi $, so that
\ekv{l2.18.1}
{
x_0^*\cdot x+x_0\cdot{2\over i}{\partial \Phi \over\partial x}(x)\in {\bf R},\
 \forall x\in {\bf C}^n.
}
By essentially the same calculation as in the real setting, we see that 
$$
(e^{i\ell})^wu(x)=e^{ix_0^*\cdot (x+{1\over 2}x_0)}u(x+x_0),\ u\in H_\Phi ,
$$
and here we recall from the unitary and metaplectic equivalence with 
$L^2({\bf R}^n)$ (via $T$) that $(e^{i\ell})^w:H_\Phi \to H_\Phi $ is unitary, or equivalently that 
\ekv{l2.18.2}
{
-\Phi (x)+\Phi (x+x_0)+\Re \Big( ix_0^*\cdot (x+{1\over 2}x_0)
\Big) =0,\ \forall x\in 
{\bf C}^n.
}
(A simple calculation shows more directly the equivalence of \no{l2.18.1} and 
\no{l2.18.2}.) Notice also that if we identify $u$ with a function 
$\widetilde{u}(\rho )$ on 
$\Lambda _\Phi $ via the natural projection $(x,\xi )\mapsto x$, then 
$u(x+x_0)$ is identified with $\widetilde{u}(\rho +H_\ell )$, where the 
Hamilton field $H_\ell $ is viewed as a real constant vector field 
on $\Lambda _\Phi $.

\par It follows that $b^w$ has a kernel satisfying
$$\vert K^{\rm eff}_{b^w}(x,y) \vert=\vert K^{\rm eff}_{a^w}(x+{1\over 2}
J{x_0^*}-x_0,
y-{1\over 2}J{x_0^*}-x_0) \vert$$
and from \no{l2.16} we get 
\ekv{l2.19}
{
\vert K^{\rm eff}_{b^w}(x,y) \vert\le C_N(a)\langle x-(x_0-{1\over 2}J{x_0^*}) \rangle 
^{-N}\langle y-(x_0+{1\over 2}J{x_0^*}) \rangle 
^{-N},
}
so the kernel of $b^w$ is concentrated near $(x_0-{1\over 2}J{x_0^*}, 
x_0+{1\over 2}J{x_0^*})$.

\par Now, let $m$ be an order function on $E\times E^*$ and let $a\in 
\widetilde{S}(m)$. Choose a lattice $\Gamma \subset E\times E^*$ and a 
 partition of unity as in \no{sy.2} as well as a function 
$\psi \in{\cal S}(E\times E^*)$ as in Lemma \ref{sy3}. Write
\ekv{l2.20}
{
a=\sum_{\gamma \in\Gamma }a_\gamma ,\ a_\gamma =\psi _\gamma ^w 
\widetilde{a}_\gamma,\ \widetilde{a}_\gamma =\chi _\gamma ^wa, 
}
where $\Vert \widetilde{a}_\gamma \Vert\le Cm(\gamma )$. Then, using that
$\psi _0^w$ is continuous: $L^2(E)\to {\cal S}(E)$, we see that
 $a_\gamma $ is 
concentrated near $\gamma $ in the above sense and more precisely, 
\ekv{l2.21}
{
\vert K^{\rm eff}_{a^w}(x,y) \vert\le C_Nm(\gamma )\langle x- 
(\gamma _x-{1\over 2}
J{\gamma_{x^*}  })\rangle ^{-N}\langle y- (\gamma _x+{1\over 2}
J{\gamma_{x^*}  })\rangle ^{-N},\ x,y\in E,
}
where we write $\gamma =(\gamma _x,\gamma _{x^*} )\in E\times E^*$. 

\par Let $q (x,y)=({x+y\over 2},J\inv (y-x))=(q_x(x,y),q_{x^*}(x,y))$, 
 so that 
$$
q\inv (\gamma )=(\gamma _x-{1\over 2}J\gamma _x,\gamma _x+
{1\over 2}J\gamma _x)
,$$
and hence
$$\langle q (x,y)-\gamma  \rangle\le {\cal O}(1)\langle x-(\gamma  _x-
{1\over 2}J{\gamma  _{x^*} }) \rangle\langle y-(\gamma  _x+
{1\over 2}J{\gamma  _{x^*} }) \rangle,$$
so \no{l2.21} implies 
\eekv{l2.22}
{
\vert K^{\rm eff}_{a_\gamma ^w}(x,y) \vert&\le& C_N(a)m(\gamma )\langle q (x,y)-
\gamma  \rangle^{-N}
}
{&\le& \widetilde{C}_N(a)m(q(x,y))\langle q (x,y)-\gamma  \rangle 
^{N_0-N},
}
where we used that $m$ is an order function in the last inequality. Choose $N$ 
with $N_0-N<-4n$, sum over $\gamma $ and use \no{l2.20} to get 
\ekv{l2.23}
{
\vert K^{\rm eff}_{a^w}(x,y) \vert \le C(a)m(q (x,y))=C(a)m({x+y\over 2}, 
J\inv (y-x)),\ x,y\in E.
}
We get
\begin{theo}\label{l21}
Let $a\in\widetilde{S}(m)$, where $m$ is an order function on $E\times E^*
$, $E=T^*{\bf R}^n$. Then $a^w$ has an effective kernel (rigorously defined 
after applying a Bargmann transform as above) satisfying \no{l2.23}, where 
$C(a)$ is a $\widetilde{S}(m)$ norm of $a$. In particular, if $M(x,y)=
m({x+y\over 2},J\inv (y-x))$ is the kernel of an $L^2(E)$-bounded operator, 
then $a^w$ is bounded: $L^2({\bf R}^n)\to L^2({\bf R}^n)$.
\end{theo}

\par As mentioned in the introduction, the statement on $L^2$-boundedness 
here is due to 
Bony \cite{Bon1}, who obtained it in a rather different way. A 
calculation, similar to the one leading to \no{l2.23}, has been given by 
Gr\"ochenig \cite{Gr2}.

\begin{cor}\label{l22}
If $M$ is the kernel of a Shur class operator i.e. if 
$$\sup_x\int m({x+y\over 2},J\inv (y-x))dy,\ \sup_y\int m({x+y\over 2},J
\inv (y-x))dx \,\, 
<\infty ,$$
then $a^w$ is bounded: $L^2({\bf R}^n)\to L^2({\bf R}^n)$.
\end{cor}
\begin{cor}\label{l23}
Assume $m(x,{x^*} )=m({x^*} )$ is independent of $x$, for $(x,{x^*} )
\in E\times E^*$ and $m({x^*} )\in L^1(E^*)$, then 
$a^w$ is bounded: $L^2({\bf R}^n)\to L^2({\bf R}^n)$.
\end{cor}

\section
{Composition}\label{co}
\setcounter{equation}{0}

Let $a,b\in {\cal S}(E)$, $E={\bf R}^n\times ({\bf R}^n)^*$, $(x_0,x_0^*),\, 
(y_0,y_0^*)\in E\times E^*$ and consider the Weyl composition of the two symbols 
$e^{x\cdot x_0^*}a(x-x_0)$, $e^{x\cdot y_0^*}b(x-y_0)$ , concentrated near 
$(x_0,x_0^*)$ and 
$(y_0,y_0^*)$ respectively:
\ekv{co.1}
{
e^{{i\over 2}\sigma (D_x,D_y)}(e^{x\cdot x_0^*}a(x-x_0)e^{y\cdot y_0^*}
b(y-y_0))(z,z).
}
We work in canonical coordinates $x\simeq (x,\xi )$ and identify $E$ and $E^*$. 
Then 
$$\sigma (x^*,y^*)=Jx^*\cdot y^*,\ J=\pmatrix{0 &1\cr -1 &0},\ {\trans J}=-J,\ 
J^2=-1,$$
and $e^{{i\over 2}\sigma (D_x,D_y)}$ is convolution with $k$, given by
$$k(x,y)={1\over (2\pi )^{2n}}\iint e^{i(x\cdot x^*+y\cdot y^*+{1\over 2}Jx^*\cdot y^*)}dx^*dy^*.$$ The phase $\Phi =x\cdot x^*+y\cdot y^*+{1\over 2}Jx^*\cdot y^*$ has a unique nondegenerate critical point $(x^*,y^*)=(2Jy,-2Jx)$ and the corresponding critical value is equal to $-2\sigma (x,y)=-2Jx\cdot y$. Hence
$k=Ce^{-2i\sigma (x,y)}=Ce^{-2iJx\cdot y}$ for some (known) constant $C$. 

\par The composition \no{co.1} becomes
\eekv{co.2}
{
&&C\iint e^{i(-2J(z-x)\cdot (z-y)+x\cdot x_0^*+y\cdot y_0^*)}a(x-x_0)b(y-y_0)dxdy=
}
{
&&Ce^{iz\cdot (x_0^*+y_0^*)}\iint 
e^{i(-2Jx\cdot y+x\cdot x_0^*+y\cdot y_0^*)}a(x+z-x_0)b(y+z-y_0)dxdy.
}
The exponent in the last integral can be rewritten as 
$$
-2Jx\cdot y+x\cdot x_0^*+y\cdot y_0^*=
-2J(x-{1\over 2}J\inv y_0^*)\cdot (y+{1\over 2}J\inv x_0^*)+{1\over 2}Jx_0^*\cdot y_0^*,
$$
and the composition \no{co.1} takes the form
$ 
e^{iz\cdot (x_0^*+y_0^*)}d(z)
$, where
$$
d(z)=Ce^{{i\over 2}\sigma (x_0^*,y_0^*)}\iint e^{-2i\sigma (x,y)}
a(x+z-(x_0+{1\over 2}
Jy_0^*))b(y+z-(y_0-{1\over 2}Jx_0^*))dxdy.
$$
Since $\sigma (x,y)$ is a nondegenerate quadratic form, we have for every $N\ge 0$
by integration by parts,
$$\vert d(z) \vert\le C_N\iint \langle (x,y)\rangle^{-N}\langle x+z-
(x_0+{1\over 2}Jy_0^*) \rangle^{-N}\langle y+z-
(y_0-{1\over 2}Jx_0^*) \rangle^{-N}dxdy.$$
Hence for every $N\ge 0$,
$$
\vert d(z) \vert\le C_N \langle z-(x_0+{1 \over 2}Jy_0^*) \rangle
^{-N}\langle z-(y_0-{1 \over 2 }Jx_0^*) \rangle ^{-N}.
$$
Using the triangle inequality, we get
$$(1+\vert z-a \vert)(1+\vert z-b\vert)\ge 1+\vert z-a \vert +\vert z-b \vert
\ge 1+{1 \over 2}\vert a-b \vert +\vert z-{a+b \over 2} \vert ,$$
so
$$
(1+\vert z-a \vert)(1+\vert z-b\vert)\ge{1 \over C}(1+\vert a-b \vert )^{1/2}
(1+\vert z-{a+b \over 2}\vert )^{1/2} 
$$
and hence for every $N\ge 0$,
\ekv{co.3}
{
\vert d(z) \vert\le C_N\langle (x_0+{1 \over 2}Jx_0^*)-(y_0-{1\over 2}Jy_0^*) 
\rangle^{-N}\langle z-{1\over 2}(x_0-{1\over 2}Jx_0^*+y_0+{1\over 2}Jy_0^*)
 \rangle^{-N}.
}
Clearly, we have the same estimates for the derivatives of $d(z)$. It follows that the composition \no{co.1} is equal to $e^{iz\cdot z_0^*}c(z-z_0)$, where 
\ekv{co.4}
{
z_0^*=x_0^*+y_0^*,\ z_0={1\over 2}(x_0-{1\over 2}Jx_0^*+y_0+{1\over 2}Jy_0^*),
}
and where $c\in{\cal S}$ and for every seminorm $p$ on ${\cal S}$ and every 
$N$, there is a seminorm $q$ on ${\cal S}$ such that
\ekv{co.5}
{
p(c)\le \langle (x_0+{1\over 2}Jx_0^*)-(y_0-{1\over 2}Jy_0^*) \rangle^{-N}q(a)q(b).
}

\par It follows that :
$$
e^{iz\cdot z_0^*}c(z-z_0)\in\widetilde{S}(\langle \cdot -(z_0,z_0^*) 
\rangle^{-M})
$$
with corresponding norm bounded by 
$$
q_{N,M}(a)q_{N,M}(b)\langle (x_0+{1\over 2}Jx_0^*)-(y_0-{1\over 2}Jy_0^*) 
\rangle ^{-N},
$$
for all $N,M\ge 0$ where $q_{N,M}$ are suitable seminorms on ${\cal S}$.

\par If $a_1\in \widetilde{S}(m_1)$, $a_2\in \widetilde{S}(m_2)$ then 
$c=a_1\# a_2$ is 
well-defined and belongs to $\widetilde{S}(m_3^{(N)})$ provided that the 
integrals defining $m_3^{(N)}$ and $m_3$ below converge. Here (replacing 
summation over lattices by integration)
\eekv{co.6}
{
m_3^{(N)}(z,z^*)&=&\iiiint \langle z^*-(x^*+y^*) \rangle^{-N} 
\langle z-{1\over 2}(x-{1\over 2}Jx^*+y+{1\over 2}Jy^*) \rangle^{-N}}
{ 
&&\times \langle (x+{1\over 2}Jx^*)-(y-{1\over 2}Jy^*)  \rangle^{-N} 
m_1(x,x^*)m_2(y,y^*)dxdydx^*dy^*
}
In order to understand the integral \no{co.6}, we put 
$\widetilde{x}={1\over 2}Jx^*$, $\widetilde{y}={1\over 2}Jy^*$, 
$\widetilde{z}={1\over 2}Jz^*$, and study the set $\Sigma (z,z^*)$ where the 
arguments inside the three brackets vanish simultaneously:
$$
\cases{\widetilde{x}+\widetilde{y}=\widetilde{z},\cr x+y-\widetilde{x}+
\widetilde{y}=2z,\cr x-y+\widetilde{x}+\widetilde{y}=0,}
$$
which can be transformed to 
\ekv{co.8}
{\Sigma (z,z^*):\
\cases{
\widetilde{x}-x=\widetilde{z}-z,\cr
\widetilde{y}+y=\widetilde{z}+z,\cr
\widetilde{x}+\widetilde{y}=\widetilde{z}.
}
}
Now it is clear that for every $M>0$ there is an $N>0$ such that
\ekv{co.9}
{
m_3^{(N)}(z,z^*)\le {\cal O}(1)\iiiint {\rm dist\,}(x,x^*,y,y^*;\Sigma (z,z^*))
^{-M}m_1(x,x^*)m_2(y,y^*)dxdydx^*dy^*.
}
Since $m_1$, $m_2$ are order functions, we have
\begin{eqnarray*}
m_1(x,x^*)&\le& {\cal O}(1){\rm dist\,}(x,x^*,y,y^*;\Sigma (z,z^*))^{N_0}m_1
(\Pi^{(1)} _\Sigma (x,x^*,y,y^*))\\
m_2(y,y^*)&\le& {\cal O}(1){\rm dist\,}(x,x^*,y,y^*;\Sigma (z,z^*))^{N_0}m_2
(\Pi^{(2)} _\Sigma (x,x^*,y,y^*)), 
\end{eqnarray*}
where $\Pi_\Sigma  :(E\times E^*)^2\to \Sigma (z,z^*)$ is the affine 
orthogonal projection and we write $\Pi _\Sigma (x,x^*;y,y^*)=
(\Pi^{(1)} _\Sigma (x,x^*;y,y^*),\Pi^{(2)} _\Sigma (x,x^*;y,y^*))
$. We conclude that for $N$ large enough,
\ekv{co.10}
{
m_3^{(N)}(z,z^*)\le {\cal O}(1)m_3(z,z^*),
}
where
\ekv{co.11}
{
m_3(z,z^*)=\int_{\Sigma (z,z^*)}m_1(x,x^*)m_2(y,y^*)d\Sigma 
}
or more explicitly,
\ekv{co.12}
{
m_3(z,z^*)=\int_{{{1\over 2}Jx^*-x={1\over 2}Jz^*-z\atop
{1\over 2}Jy^*+y={1\over 2}Jz^*+z
}\atop x^*+y^*=z^*}m_1(x,x^*)m_2(y,y^*))dx.
}
Reversing the above estimates, we see that $m_3(z,z^*)\le {\cal O}
(1)m_3^{(N)}(z,z^*)$, if $N>0$ is large enough.

\begin{prop}\label{co1}
If the integral in \no{co.11} converges for one value of $(z,z^*)$, then it 
converges for all values and defines an order function $m_3$.
\end{prop}
\begin{proof}
Suppose the integral converges for the value $(z,z^*)$ and consider any other 
value $(z+t,z^*+t^*)$. We have the measure preserving map
$$
\Sigma (z,z^*)\ni (x,x^*,y,y^*)\mapsto (x+t,x^*+t^*,y+{1\over 2}Jt^*+t,y^*)
\in \Sigma (z+t,z^*+t^*),
$$
so
\begin{eqnarray*}
m_3(z+t,z^*+t^*)&=&\int_{\Sigma (z,z^*)}m_1(x+t,x^*+t^*)
m_2(y+{1\over 2}Jt^*+t,y^*)dx\\
&\le& C\langle (t,t^*) \rangle^{N_0}\langle t+{J\over 2}t^* \rangle^{N_0} 
m_3(z,z^*)\\
&\le& \widetilde{C}\langle (t,t^*) \rangle^{2N_0}m_3(z,z^*).
\end{eqnarray*}
The proposition follows.
\end{proof}

\par From the above discussion, we get
\begin{theo}\label{co2}
Let $m_1$, $m_2$ be order functions on $E\times E^*$ and define $m_3$ by \no{co.12}. Assume that $m_3(z,z^*)$ is finite for at least one $(z,z^*)$ so that 
$m_3$ is a well-defined order function by Proposition \ref{co1}. Then the 
composition map
\ekv{co.13}{
{\cal S}(E)\times{\cal S}(E)\ni (a_1,a_2)\mapsto a_1\# a_2\in {\cal S}(E) 
}
has a bilinear  extension
\ekv{co.14}{
\widetilde{S}(m_1)\times \widetilde{S}(m_2)\ni (a_1,a_2)\mapsto a_1\# a_2\in \widetilde{S}(m_3),}
Moreover,
\ekv{co.15}
{
\Vert a_1\# a_2\Vert_{\widetilde{S}(m_3)}\le {\cal O}(1)
\Vert a_1\Vert_{\widetilde{S}(m_1)}\Vert a_2\Vert_{\widetilde{S}(m_2)}.
}
\end{theo}

\par We end this section by establishing a connection with the effective 
kernels of Section \ref{l2}. Let $a_j$ be as in the theorem with 
$a_3=a_1\# a_2$. According to Theorem \ref{l21}, we then know that $a_j^w$ 
has an effective kernel $K_j=K^{\rm eff}_{a_j^w}(x,y)$ satisfying 
\ekv{co.15.5}
{
K_j(x,y)={\cal O}(1)m_j(q (x,y)),\hbox{ where } q (x,y)=
({x+y\over 2},J\inv (y-x)).
}
Since the composition of the effective kernels of $a_1^w$ and $a_2^w$ is an 
effective kernel for $a_3^w=a_1^w\circ a_2^w$ we expect that 
\ekv{co.16}
{
m_3(q (\widetilde{x},\widetilde{y}))= C\int 
m_1(q (\widetilde{x},\widetilde{z}))m_2(q (\widetilde{z},
\widetilde{y}))d\widetilde{z},
}
or more explicitly,
\ekv{co.17}
{
m_3({\widetilde{x}+\widetilde{y}\over 2},J\inv (\widetilde{y}-\widetilde{x}))
=C\int 
m_1({\widetilde{x}+\widetilde{z}\over 2},J\inv (\widetilde{z}-\widetilde{x}))
m_2({\widetilde{z}+\widetilde{y}\over 2},J\inv (\widetilde{y}-\widetilde{z}))
d\widetilde{z},
}
Writing 
\begin{eqnarray*}
z&=&{\widetilde{x}+\widetilde{y}\over 2},\\
z^*&=&J\inv (\widetilde{y}-\widetilde{x}),\\
x&=&{\widetilde{x}+\widetilde{z}\over 2},\\
x^*&=&J\inv (\widetilde{z}-\widetilde{x}),\\
y&=&{\widetilde{z}+\widetilde{y}\over 2},\\
y^*&=&J\inv (\widetilde{y}-\widetilde{z}),\\
\end{eqnarray*} 
we check that the integral in \no{co.17} coincides with the one in \no{co.12} 
up to a constant Jacobian factor, so the results of this section fit 
with the ones of Section \ref{l2}.

\begin{ex}\label{co3}\rm
Let $a_j\in \widetilde{S}(m_j)$, $j=1,2$, where $m_j$ are order functions on $E\times E^*$ of the form
\begin{eqnarray*}
m_j(x,x^*)&=&\widetilde{m}_j(x)\langle x^* \rangle^{-N_j},\ N_j\in{\bf R},\\
\widetilde{m}_j(x)&\le&C\langle x-y \rangle^{M_j}\widetilde{m}_j(y),\ 
x,y\in E,\ M_j\ge 0.
\end{eqnarray*}
Then, the effective kernels $K_1,K_2$ of $a_1^w, a_2^w$ satisfy
$$
K_j(x,y)={\cal O}(1)m_j({x+y\over 2},J\inv (y-x))={\cal O}(1)
\widetilde{m}_j({x+y\over 2})\langle x-y \rangle^{-N_j}.
$$
Then $a_1\# a_2$ is well-defined and belongs to $\widetilde{S}(m_3)$, where 
$$
m_3({x+y\over 2},J\inv (y-x))=\int \widetilde{m}_1({x+z\over 2})
\langle x-z \rangle^{-N_1}\langle z-y \rangle^{-N_2}
\widetilde{m}_2({z+y\over 2}) dz,
$$
provided that the last integral converges for at least one (and then all) 
value(s) of $((x+y)/2,J\inv (y-x))$. If we use that 
\begin{eqnarray*}
\widetilde{m}_1({x+z\over 2})&\le& {\cal O}(1)\widetilde{m}_1(
{x+y\over 2})\langle z-y \rangle^{M_1}\\
\widetilde{m}_2({z+y\over 2})&\le& {\cal O}(1)\widetilde{m}_2(
{x+y\over 2})\langle x-z \rangle^{M_2},
\end{eqnarray*}
we get 
\ekv{co.18}
{
m_3({x+y\over 2},J\inv (y-x))\le {\cal O}(1)\widetilde{m}_1({x+y\over 2})
\widetilde{m}_2({x+y\over 2})\int \langle x-z \rangle^{-N_1+M_2}
\langle z-y \rangle^{-N_2+M_1}dz.
}
Thus $m_3$ and $a_1\#a_2\in\widetilde{S}(m_3)$ are well-defined if 
\ekv{co.19}
{
-(N_1+N_2)+M_1+M_2<-2n.
}
The integral $I$ in \no{co.18} is ${\cal O}(1)$ in any region where 
$x-y={\cal O}(1)$. For $\vert x-y \vert \ge 2$, we write 
$I\le I_1+I_2+I_3$, where 
\begin{itemize}
\item $I_1$ is the integral over $\vert x-z \vert\le {2\over 3}
\vert x-y \vert$. Here $\langle z-y \rangle\backsim \langle x-y \rangle$.
\item $I_2$ is the integral over $\vert z-y \vert\le {2\over 3}
\vert x-y \vert$. Here $\langle x-z \rangle\backsim \langle x-y \rangle$.
\item $I_3$ is the integral over $\vert x-z \vert, \vert z-y \vert
\ge {2\over 3}
\vert x-y \vert$. Here $\langle x-z \rangle\backsim \langle y-z \rangle\ge 
{1\over C}\langle x-y \rangle$.
\end{itemize}

\par We get
$$
I_1\backsim \langle x-y \rangle^{-N_2+M_1}\int_0^{\langle x-y \rangle}
\langle r \rangle^{-N_1+M_2+2n-1}dr\backsim \langle x-y \rangle
^{-N_2+M_1+(-N_1+M_2+2n)_+},
$$
with the convention that we tacitly add a factor $\ln \langle x-y \rangle$ 
when the expression inside $(..)_+$ is equal to $0$. Similarly (with the 
same convention),
$$
I_2\backsim \langle x-y \rangle
^{-N_1+M_2+(-N_2+M_1+2n)_+}.
$$ 
In view of \no{co.19}, we have 
$$
I_3\backsim \int_{\langle x-y \rangle}^\infty r^{-(N_1+N_2)+M_1+M_2+2n-1}dr
\backsim \langle x-y \rangle^{-(N_1+N_2)+M_1+M_2+2n}.
$$
it follows that 
\ekv{co.19.5}
{
I\backsim \langle x-y \rangle^{\max (-N_2+M_1+(-N_1+M_2+2n)_+,-N_1+M_2
+(-N_2+M_1+2n)_+)},
}
so with the same convention, we have
\ekv{co.20}
{m_3(x,x^*)\le {\cal O}(1)\widetilde{m}_1(x)\widetilde{m}_2(x)
\langle x^* \rangle^{\max (-N_2+M_1+(-N_1+M_2+2n)_+,-N_1+M_2
+(-N_2+M_1+2n)_+)}.
}
This simplifies to 
\ekv{co.21}
{
m_3(x,x^*)\le {\cal O}(1)\widetilde{m}_1(x)\widetilde{m}_2(x)
\langle x^* \rangle^{\max (-N_2+M_1,-N_1+M_2)}
}
if we strengthen the assumption \no{co.19} to:
\ekv{co.22}
{
-N_1+M_2,\, -N_2+M_1<-2n.
}
\end{ex}

\section{More direct approach using Bargmann transforms}\label{di}
\setcounter{equation}{0}

\par By using Bargmann transforms more systematically (from the point of 
view of Fourier integral operators with complex 
phase) the results of Section \ref{l2}, \ref{co} can be obtained more directly.
The price to pay however, is the loss of some aspects that might be 
helpful in other situations like the ones with variable metrics.

\par Let $F$ be  real $d$-dimensional space as in Section \ref{sy} and 
define $T:L^2(F)\to H_\Phi (F^{\bf C})$ as in \no{l2.8}--\no{l2.11}. Then 
we have 
\begin{prop}\label{di1}
If $m$ is an order function on $F\times F^*$, then 
\ekv{di.1}
{
\widetilde{S}(m)=\{ u\in{\cal S}'(F);\, e^{-\Phi (x)}\vert Tu(x) \vert \le C 
m(\kappa _T\inv (x,{2\over i}{\partial \Phi \over\partial x}(x)))\} ,
}
where the best constant $C=C(m)$ is a norm on $\widetilde{S}(m)$. 
\end{prop}
\begin{proof}
Assume first that $u$ belongs to $\widetilde{S}(m)$ and write 
$ u=\sum_{\gamma \in \Gamma }\psi _\gamma ^w \chi_\gamma ^w u$ as in Lemma \ref{sy3}. 
The effective kernel of $\psi _\gamma ^w$ satisfies
\ekv{di.1.5}
{
|K^{\rm eff}_{\psi _\gamma ^w}(x,y)|\le C_N \langle x-\gamma \rangle ^{-N}
\langle y-\gamma \rangle ^{-N},
}
for every $N>0$, where throughout the proof we identify $F^{\bf C}$ with 
$F\times F^*$ by means of $\pi \circ \kappa _T$ and work on the latter space.
Here $\pi :\Lambda _\Phi \to F^{\bf C}$ is the natural projection. 
Then we see that 
$$
|e^{-\Phi /h}Tu(x)|\le C_N(u)\sum_{_\gamma \in \Gamma }m(\gamma )
\langle x-\gamma \rangle ^{-N}={\cal O}(m(x)).
$$

\par Conversely, if $e^{-\Phi /h}Tu={\cal O}(m(x))$, then since the effective 
kernel of $\chi_\gamma ^w$ also satisfies \no{di.1.5}, we see that 
$e^{-\Phi /h}T\chi_\gamma ^wu={\cal O}_N(\langle x-\gamma \rangle ^{-N}m(\gamma ))$, 
implying $\Vert e^{-\Phi /h}T\chi_\gamma ^wu \Vert_{L^2}={\cal O}(m(\gamma ))$, and 
hence $\Vert \chi_\gamma ^wu\Vert ={\cal O}(m(\gamma )).$
\end{proof}

\par With this in mind, we now take $a\in \widetilde{S}({\bf R}^n\times 
({\bf R}^n)^*;m)$ and look for an explicit choice of effective kernel for 
$a^w$.
Let $T:L^2({\bf R}^n)\to H_\Phi ({\bf C}^n)$ be a Bargmann transform as above. Consider first the map $a\mapsto K_{a^w}(x,y)\in {\cal S}'({\bf R}^n\times 
{\bf R}^n)$ from $a$ to the distribution kernel of $a^w$, given by 
\eekv{di.2}
{
K_{a^w}(x,y)&=&{1\over (2\pi )^n}\int e^{i(x-y)\cdot \tau} a({x+y\over 2},\tau )
d\tau}{&=& {1\over (2\pi )^{2n}}\iiint e^{i(x-y)\cdot \tau +i({x+y\over 2}-t)
\cdot s} a(t,\tau )dtdsd\tau .
} 
We view this as a Fourier integral operator $B:\, a\mapsto K_{a^w}(x,y)$ with 
quadratic phase. The associated linear \ctf{} is given by:
$$
\kappa _B:\ (t,\tau; t^*,\tau^*)=({x+y\over 2},\tau ;s,y-x)\mapsto (x,\tau+{s\over 2};y,-\tau+
{s\over 2})=(x,x^*;y,y^*),
$$
which we can write as
\ekv{di.3}
{\kappa _B:\ (t,\tau ;t^*,\tau^*)\mapsto (t-{\tau^*\over 2},\tau+{t^*\over 2};
t+{\tau^*\over 2},-\tau+{t^*\over 2}).
}

\par From the unitarity of $T$, we know that $T^*T=1$, where 
\ekv{di.4}
{
T^*v(y)=C\int e^{-i\overline{\phi (x,y)}}v(x)e^{-2\Phi (x)}L(dx).
}
We can therefore define the effective kernel of $a^w$ to be 
\ekv{di.4.5}{K^{\rm eff}(x,y)=e^{-\Phi (x)}
K(x,\overline{y})e^{-\Phi (y)},} 
where
\eekv{di.5}
{
Ta^wT^*v(x)&=&\int K(x,\overline{y})v(y)e^{-2\Phi (y)}L(dy),\ v\in 
H_\Phi ({\bf C}^n),
}
{
K(x,\overline{y})&=& C^2 \iint e^{i(\phi (x,t)-\overline{\phi (y,s)})}
K_{a^w}(t,s)dtds.
}
We write this as
$$
K(x,y)=C^2\iint e^{i(\phi (x,t)-\phi ^*(y,s))}K_{a^w}(t,s)dtds,
$$
with $\phi ^*(y,s)=\overline{\phi (\overline{y},\overline{s})}$, so 
\ekv{di.6}
{K(x,y)=(T\otimes \widetilde{T})(K_{a^w})(x,y),
}
where 
\ekv{di.7}
{
(\widetilde{T}u)(y)=C\int e^{-i\phi ^*(y,s)}u(s)ds
=\overline{(T\overline{u})(\overline{y})}
.
}

\par We see that $\widetilde{T}:L^2({\bf R}^n)\to H_{\Phi ^*}({\bf C}^n)$ is 
a unitary Bargmann transform, where
\ekv{di.8}
{
\Phi ^*(y)=\sup_{s\in{\bf R}^n}\Im \phi ^*(y,s)=\sup_{s\in{\bf R}^n}\Im 
\overline{\phi (\overline{y},s)}=\Phi (\overline{y}).
}

The \ctf{} associated to $\widetilde{T}$ is 
\ekv{di.9}
{
\kappa _{\widetilde{T}}:\, (s,{\partial \phi ^*\over \partial s}(y,s))
\mapsto (y,-{\partial \phi ^*\over \partial y}(y,s)).
}
If 
\ekv{di.10}
{
\iota (s,\sigma )=(\overline{s},-\overline{\sigma }),
}
we check that 
\ekv{di.11}
{
\kappa _{\widetilde{T}}=\iota \kappa _T\iota ,\quad \iota :(x,{2\over i}
{\partial\Phi  \over\partial x}(x))\mapsto (\overline{x},{2\over i}
{\partial\Phi^*  \over\partial y}(\overline{x})).
}
\par Clearly $T\otimes \widetilde{T}$ is a Bargmann transform 
with associated \ctf{} $\kappa _T\times (\iota \kappa _T\iota ) $, 
so in view of \no{di.3} the map $a\mapsto K$ is also a Bargmann 
transform with associated canonical 
transformation 
\ekv{di.12}
{
(E\times E^*)^{\bf C}\ni (t,\tau ;t^*,\tau^*)\mapsto (\kappa _T((t,\tau )-
{1\over 2}J(t^*,\tau ^*)),\iota \kappa _T((\overline{(t,\tau )}+{1\over 2}
\overline{J(t^*,\tau^*)})),
}
where $E={\bf R}^n\times ({\bf R}^n)^*$.
The restriction to the real phase space is
\eekv{di.13}
{
&&E\times E^*\ni (t,\tau ;t^*,\tau^*)\mapsto} 
{&&(\kappa _T((t,\tau )-
{1\over 2}J(t^*,\tau ^*)),\iota \kappa _T(({(t,\tau )}+{1\over 2}{(t^*,\tau^*)}))\in \Lambda _\Phi \times \iota \Lambda _\Phi =\Lambda _\Phi \times \Lambda _{\Phi ^*},
}
and this restriction determines our complex linear canonical 
transformation uniquely.
 
\par As in Section \ref{l2} we may view the effective kernel 
$K^{\rm eff}(x,y)$ 
in \no{di.4.5} as a function on $E\times E$, by identifying $x,y\in {\bf C}^n$ 
with $\kappa _T\inv (x,{2\over i}{\partial \Phi \over \partial x}(x)),\, 
\kappa _T\inv 
(y,{2\over i}{\partial \Phi \over \partial x}(y))\in E$ respectively. With this 
identification and using also the general characterization in \no{di.1} (with
$T$ replaced by $T\otimes\widetilde{T})$), 
we see that if $a\in{\cal S}'(E)$, then $a\in\widetilde{S}(m)$ iff
\ekv{di.14}
{
K^{\rm eff}(t-{1\over 2}Jt^*,t+{1\over 2}Jt^*)={\cal O}(1)m(t,t^*),\ (t,t^*)\in E\times E^*,
} 
where we shortened the notation by writing $t$ instead of $(t,\tau )$ and $t^*$ 
instead of $(t^*,\tau ^*)$.

\par Theorem \ref{l21} now follows from \no{di.14}, \no{di.4.5}, \no{di.5}. 

\par Theorem \ref{co2} also follows from  \no{di.14}, \no{di.4.5}, \no{di.5} 
together with
the remark that the kernel $K(x,y)=K_a(x,y)$ is the unique kernel which is 
holomorphic on ${\bf C}^n\times {\bf C}^n$, 
such that the corresponding $K^{\rm eff}_{a^w}$ given in \no{di.4.5} is of 
temperate growth at infinity and \no{di.5} 
is fulfilled. Indeed, then it is clear that 
\ekv{di.15}
{
K^{\rm eff}_{(a_1\#_2)^w}(x,y)=\int K^{\rm eff}_{a_1^w}(x,z)
K^{\rm eff}_{a_2^w}
(z,y) L(dz) } and the bound \no{di.14} for $a_1\# a_2$ with $m=m_3$ follows directly from the corresponding bounds for $a_j$ with $m=m_j$.

\section{$C_p$ classes}\label{cp}
\setcounter{equation}{0}

\par In this section we give a simple condition on an order function $m$
on $E\times E^*$ ($E=T^*{\bf R}^n$) and 
a number $p\in [1,\infty ]$ that implies the property:
\ekv{cp.1}{
\exists C>0\hbox{ such that: }a\in \widetilde{S}(m)\Rightarrow a^w\in 
C_p(L^2,L^2)\hbox{ and }\Vert a^w\Vert_{C_p}\le C
\Vert a\Vert_{\widetilde{S}(m)}.
}
Here $C_p(L^2,L^2)$ is the Schatten--von Neumann class of operators: $L^2(
{\bf R}^n)\to L^2({\bf R}^n)$, see for instance \cite{GoKr}. 

Let $m$ be an order function on $E\times E^*$ and let $p\in [1,+\infty ]$. 
Consider the following property, where $q$ is given in \no{co.15.5} and
$\Gamma \subset E$ is a lattice, 
\eekv{cp.2}
{&&
\exists C>0\hbox{ such that if }\vert a_{\alpha ,\beta } \vert\le 
m(q (\alpha ,\beta )),\ \alpha ,\beta \in\Gamma ,
}
{&&
\hbox {then } (a_{\alpha ,\beta })_{\alpha, \beta\in \Gamma }\in 
C_p(\ell^2(\Gamma ), 
\ell^2(\Gamma ))\hbox{ and } \Vert (a_{\alpha ,\beta })\Vert_{C_p}\le C.
}
Notice that if \no{cp.2} holds and if we fix some number 
$N_0\in {\bf N}^*$, then
if $(A_{\alpha ,\beta })_{\alpha ,\beta \in \Gamma }$  
is a block matrix where every $A_{\alpha ,\beta }$ is an $N_0\times N_0$
matrix then
\ekv{cp.3}
{
\hbox{same as \no{cp.2} with }a_{\alpha ,\beta }\hbox{ replaced by }
A_{\alpha ,\beta } \hbox{ and }\vert \cdot  \vert \hbox{ by }
\Vert \cdot \Vert _{{\cal L}({\bf C}^{N_0},{\bf C}^{N_0})}.
}
\begin{prop}\label{cp1}
The property \no{cp.2} only depends on $m,p$ but not on the choice of $\Gamma $.
\end{prop}
\begin{proof}
Let $m,p,\Gamma $ satisfy \no{cp.2} and let $\widetilde{\Gamma }$ be a 
second lattice in $E$. Let $(a_{\widetilde{\alpha },\widetilde{\beta }})$ 
be a $\widetilde{\Gamma }\times \widetilde{\Gamma}$ matrix satisfying 
$\vert a_{\widetilde{\alpha },\widetilde{\beta }} \vert \le 
m(q (\widetilde{\alpha },\widetilde{\beta }))$.
Let $\pi (\widetilde{\alpha })\in \Gamma $ be a point that realizes the 
distance from $\widetilde{\alpha } $ to $\Gamma $, so that 
$\vert \pi (\widetilde{\alpha })-\widetilde{\alpha } \vert\le C_0$ for 
some constant $C_0>0$. Let $N_0=\max \#\pi \inv (\alpha )$ and choose an 
enumeration $\pi\inv (\alpha )=\{ \widetilde{\alpha }_1,...,
\widetilde{\alpha }_{N(\alpha )}\}$, $N(\alpha )\le N_0$, for 
every $\alpha \in\Gamma $. 
Then we can identify $(a_{\widetilde{\alpha },\widetilde{\beta }})_
{\widetilde{\Gamma }\times\widetilde{\Gamma }}$ with the matrix 
$(A_{\alpha ,\beta })_{\alpha ,\beta \in \Gamma \times \Gamma }$ where 
$A_{\alpha ,\beta }$ is the $N_0\times N_0$ matrix with the entries
$$
(A_{\alpha ,\beta })_{j,k}=\cases{a_{\widetilde{\alpha }_j,\widetilde{\beta }_k},
\hbox{ if } 1\le j\le N(\alpha ),\ 1\le k\le N(\beta ),\cr
0, \hbox{ otherwise.}}
$$
Then $\Vert A_{\alpha ,\beta }\Vert \le Cm(q (\alpha ,\beta ))$ and we 
can apply \no{cp.3} to conclude.
\end{proof}

\begin{theo}\label{cp2}
Let $m$ be an order function and $p\in [1,\infty ]$. If \no{cp.2} holds, 
then we have \no{cp.1}.
\end{theo}

\begin{proof}
Assume that \no{cp.2} holds and
let $a\in \widetilde{S}(m)$. Define $K(x,\overline{y})$ as in \no{di.5}. It 
suffices to estimate the $C_p$ norm of the operator 
$A:L^2(e^{-2\Phi }L(dx))\to L^2(e^{-2\Phi }L(dx))$, given by 
$$
Au(x)=\int K(x,\overline{y})u(y)e^{-2\Phi (y)}L(dy),
$$ 
or equivalently the one of $A_{{\rm eff}}:L^2({\bf C}^n)\to L^2({\bf C}^n)$, 
given by
\ekv{cp.3.5}{
A_{{\rm eff}}u(x)=\int K^{{\rm eff}}(x,y)u(y)L(dy),
}
with $K^{\rm eff}$ given in \no{di.4.5}.
Recall that $K^{\rm eff}(x,y)={\cal O}(1)m(q (x,y))$ 
(identifying ${\bf C}^n$ with $T^*{\bf R}^n$ via $\pi _x\circ \kappa _T$),
so $K(x,\overline{y})={\cal O}(1)m(q (x,y))e^{\Phi (x)+\Phi (y)}$. 

\par For $\alpha ,\beta \in \Gamma $ we have (identifying $\Gamma $ with a lattice in ${\bf C}^n$)
\ekv{cp.4}
{
K(x,\overline{y})=e^{F_\alpha (x-\alpha )}\widetilde{K}_{\alpha ,\beta }
(x,\overline{y})e^{\overline{F_\beta (y-\beta )}},
}
where 
\ekv{cp.5}
{F_\alpha (x-\alpha )=\Phi (\alpha )+2{\partial \Phi \over \partial x}
(\alpha )\cdot (x-\alpha )}
is \hol{} with 
\ekv{cp.6}
{
\Re F_\alpha (x-\alpha )=\Phi (x)+R_\alpha (x-\alpha ),\ 
R_\alpha (x-\alpha )={\cal O}(\vert x-\alpha  \vert^2),
}
and 
\ekv{cp.7}
{
\vert \nabla_x ^k\nabla_y ^\ell \widetilde{K}_{\alpha ,\beta }(x,\overline{y}) 
\vert \le \widetilde{C}_{k,\ell}\, m(q (\alpha ,\beta )),\ 
\vert x-\alpha  \vert, \vert y-\beta  \vert \le C_0.
}
Here we identify $\alpha ,\beta \in E$ with their images 
$\pi _x\kappa _T(\alpha ), \pi _x\kappa _T(\beta )\in {\bf C}^n
$ respectively. In fact, the case $k=\ell =0$ is clear and we get the extension to 
\aby{} $k,\ell $ from the Cauchy inequalities, since 
$\widetilde{K}_{\alpha ,\beta }$ 
is \hol{}.

\par We can also write
\ekv{cp.8}
{
K^{\rm eff}(x,y)=e^{iG_\alpha (x-\alpha )}K_{\alpha ,\beta }(x,y)
e^{-iG_\beta (y-\beta )},
}
where $$G_\alpha (x-\alpha )=\Im F_\alpha (x-\alpha ),\quad 
K_{\alpha ,\beta }=e^{R_\alpha (x-\alpha )}\widetilde{K}_{\alpha ,\beta }
(x,\overline{y})e^{R_\beta (y-\beta )},
$$
so 
\ekv{cp.9}
{
\vert \nabla_x ^k\nabla_y ^\ell {K}_{\alpha ,\beta }(x,y) 
\vert \le {C}_{k,\ell}m(q (\alpha ,\beta )),\ 
\vert x-\alpha  \vert, \vert y-\beta  \vert \le C_0.
}

\par Consider a partition of unity
\ekv{cp.10}
{ 1=\sum_{\alpha \in\Gamma  }\chi_\alpha (x),\quad \chi_\alpha
(x)=\chi_0(x-\alpha ),\
\chi_0\in C_0^\infty (\Omega _0;{\bf R}),
} where $\Omega _0$ is open with smooth boundary. Let $\Omega _\alpha =
\Omega _0+\alpha $, so that \no{cp.9} holds for $(x,y)\in \Omega _\alpha \times 
\Omega _\beta $.

\par Let $W:L^2({\bf C}^n)\to\bigoplus_{\beta \in\Gamma }L^2(\Omega _\beta )$ be defined by 
$$
Wu=\Big( {{(e^{-iG_\beta (x-\beta )}u(x))}_\vert}_{\Omega _\beta
}\Big) _{\beta \in\Gamma },
$$
so that the adjoint of $W$ is given by 
$$
W^*v=\sum_{\alpha \in\Gamma }e^{iG_\alpha (x-\alpha )}v_\alpha (x)
1_{\Omega _\alpha }(x),\quad v=(v_\alpha )_{\alpha \in\Gamma }\in
\bigoplus_{\alpha \in \Gamma }L^2(\Omega _\alpha )
.
$$
Then $W$ and its adjoint are bounded operators and 
\ekv{cp.11}
{
A_{\rm eff}=W^*{\cal A}W,
}
where ${\cal A}=(A_{\alpha ,\beta })_{\alpha ,\beta \in\Gamma }$ and 
$A_{\rm eff}:L^2({\bf C}^n)\to L^2({\bf C}^n)$, $A_{\alpha ,\beta }:
L^2(\Omega_\beta )\to L^2(\Omega _\alpha )$ are given by the kernels 
$K^{\rm eff}(x,y)$ and $\chi_\alpha (x)K_{\alpha ,\beta }(x,y)
\chi_\beta (y)$ respectively. It now suffices to show that 
$$
{\cal A}:\bigoplus_{\beta \in \Gamma }L^2(\Omega _\beta )\to
\bigoplus_{\beta \in \Gamma }L^2(\Omega _\beta )
$$ 
belongs to $C_p$ with a norm 
that is bounded by a constant times the $\widetilde{S}(m)$-norm of $a$.

\par Let $e_0,e_1,..\in L^2(\Omega _0)$ be an orthonormal basis of 
eigenfunctions of minus the Dirichlet Laplacian in $\Omega _0$, arranged 
so that the corresponding eigenvalues form an increasing sequence. Then 
$e_{\alpha ,j}:=\tau_\alpha e_j$, $j=0,1,...$ form an orthonormal basis of 
eigenfunctions of the corresponding operator in $L^2(\Omega _\alpha )$. 
From \no{cp.9} it follows that the matrix elements $K_{\alpha ,j;\beta ,k}$ 
of $A_{\alpha ,\beta }$ with respect to the bases $(e_{\alpha ,\cdot })$ 
and $(e_{\beta ,\cdot })$ satisfy
\ekv{cp.12}
{
\vert K_{\alpha ,j;\beta ,k}\vert \le C_Nm(q (\alpha ,\beta ))
\langle j \rangle^{-N}\langle k \rangle^{-N},
}  
for every $N\in{\bf N}$. We notice that $(K_{\alpha ,j;\beta ,k})
_{(\alpha ,j),(\beta ,k)\in \Gamma \times {\bf N}}$ is the matrix of 
${\cal A}$ with respect to the orthonormal basis 
$(e_{\alpha ,j})_{(\alpha ,j)\in\Gamma 
\times {\bf N}}$. We can represent this matrix as a block matrix 
$(K^{j,k})_{j,k\in{\bf N}}$, where $K^{j,k}:\ell^2(\Gamma )\to\ell^2(\Gamma )$ 
has the matrix $(K_{\alpha ,j;\beta ,k})
_{\alpha ,\beta \in\Gamma }$. Since \no{cp.2} holds and $a\in\widetilde{S}(m)$, we 
deduce from \no{cp.12} that 
\ekv{cp.12.4}
{
\Vert K^{j,k}\Vert_{C_p}\le \widetilde{C}_N\langle j \rangle^{-N}
\langle k \rangle^{-N}.
}
Choosing $N>2n$, we get 
\ekv{cp.12.5}{
\Vert {\cal A}\Vert_{C_p}\le \sum_{j,k}\Vert K^{j,k}\Vert_{C_p}<\infty .
} 
Hence $a^w\in C_p$ and the uniform bound $\Vert a^w\Vert_{C_p}\le 
\Vert a\Vert_{\widetilde{S}(m)}$ also follows from the proof.
\end{proof}

\begin{ex}\label{cp3}\rm
Assume that 
\ekv{cp.13}
{\int_{E^*} \Vert m(\cdot ,x^*)\Vert _{L^p(E)}dx^*<\infty .}
\end{ex}
Then 
\ekv{cp.14}
{\Big( m(q (\alpha, \beta))\Big)_{\alpha ,\beta \in\Gamma }=
\Big( m({\alpha +\beta \over 2},
J\inv (\beta -\alpha ))\Big)_{\alpha ,\beta \in\Gamma }} 
is a matrix where each translated diagonal $\{ (\alpha ,\beta )\in \Gamma 
\times \Gamma ;\, 
\alpha -\beta =\delta\} $ has an $\ell ^p$ norm which is summable with 
respect to $\delta \in \Gamma $. Now a matrix with non-vanishing elements 
in only one translated diagonal has a $C_p$ norm equal to the $\ell ^p$ 
norm of that diagonal, so we conclude that the $C_p$ norm of the matrix 
in \no{cp.14} is bounded by 
$$
\sum_{\delta \in \Gamma }\Vert m({\cdot \over 2},\delta )\Vert_{\ell^p}<\infty .
$$
We clearly have the same conclusion for every matrix 
$(a_{\alpha ,\beta })_{\alpha ,\beta \in\Gamma }$ satisfying 
$\vert a_{\alpha ,\beta }  \vert\le m(q (\alpha, \beta))$, so 
 \no{cp.2} holds and hence by Theorem \ref{cp2} we have the 
property \no{cp.1}.

\section{Further generalizations}\label{ge}
\setcounter{equation}{0}

\par Let $E$ be a $d$-dimensional real vector space and let 
$\Gamma \subset E$ be a lattice. We shall extend the preceding results
by replacing the $\ell^\infty (\Gamma )$-norm in the definition of 
the symbol spaces by a more general Banach space norm. Let 
$B$ be a Banach space of functions $u:\Gamma \to {\bf C}$ with the 
following properties:
\ekv{ge.1}
{
\hbox{If }u\in B,\ \gamma \in \Gamma ,\hbox{ then }\tau_\gamma u\in B,
\hbox{ and }\Vert \tau_\gamma u\Vert_B=\Vert u\Vert_B .
}
\ekv{ge.2}
{
\delta _\gamma \in B,\ \forall \gamma \in \Gamma ,
}
where $\tau_\gamma u(\alpha )=u(\alpha -\gamma )$, 
$\delta _\gamma (\alpha )=\delta _{\gamma ,\alpha }$, 
$\alpha \in 
\Gamma  $. (The last assumption will soon be replaced by a stronger one.)

\par If $u=\sum_{\gamma \in \Gamma } u(\gamma )\delta _\gamma \in B$, we get
$$
\Vert u\Vert_{B}\le \sum \vert u(\gamma ) \vert\Vert \delta _\gamma \Vert_B= C\Vert u\Vert_{\ell^1},
$$
where $C=\Vert \delta _\gamma \Vert_B$ (is independent of $\gamma $). Thus 
\ekv{ge.3}{\ell^1(\Gamma )\subset B.}

\par We need to strengthen \no{ge.2} to the following assumption:
\eekv{ge.4}
{&&
\hbox{If }u\in B\hbox{ and }v:\Gamma \to {\bf C} \hbox{ satisfies }
\vert v(\gamma ) \vert\le \vert u(\gamma ) \vert ,\ \forall \gamma \in \Gamma ,
}
{&&\hbox{then }v\in B\hbox{ and }\Vert v\Vert_B\le C\Vert u\Vert_B,\hbox{ where }
C\hbox{ is independent of }u,v.
}
It follows that $\Vert u(\gamma )\delta _\gamma \Vert_B\le C\Vert u\Vert_B$, 
for all $u\in B$, $\gamma \in \Gamma $, or equivalently that 
$$
\vert u(\gamma ) \vert\le {C\over \Vert \delta _\gamma \Vert_B}
\Vert u\Vert_B=\widetilde{C}\Vert u\Vert_B ,
$$
so 
\ekv{ge.5}
{
B\subset \ell^\infty (\Gamma ),\hbox{ and }\Vert u\Vert_{\ell^\infty }
\le \widetilde{C}\Vert u\Vert_B,\ \forall u\in B.
}

\par If $f\in\ell^1(\Gamma )$ then using only the translation invariance
\no{ge.1}, we get
\ekv{ge.6}
{
u\in B\Rightarrow \cases{f*u\in B,\cr \Vert f*u\Vert_B\le \Vert f\Vert_{\ell^1}
\Vert u\Vert_B.}
}

\par Using also \no{ge.4} we get the following partial strengthening: Let 
$k:\, \Gamma \times \Gamma \to \Gamma $ satisfy $\vert k(\alpha ,\beta ) 
\vert \le f(\alpha -\beta )$ where $f\in \ell^1(\Gamma )$. Then
\ekv{ge.7}
{
u\in B\Rightarrow v(\alpha ):=\sum_{\beta \in\Gamma }k(\alpha ,\beta )
u(\beta ) \in B\hbox{ and }\Vert v\Vert_B\le C\Vert f\Vert_{\ell^1}
\Vert u\Vert_B,
}
where $C$ is independent of $k,u$. In fact,
$$
u\in B\Rightarrow \vert u \vert \in B\Rightarrow f*\vert u \vert\in B,
$$
and $v$ in \no{ge.7} satisfies $\vert v \vert\le f*\vert u \vert$ pointwise.

\par Let $\widetilde{\Gamma }\subset E$ be a second lattice and let 
$\widetilde{B}
\subset\ell^\infty (\widetilde{\Gamma })$ satisfy \no{ge.1}, \no{ge.4}. 
We say that $B\prec \widetilde{B}$ if the following property holds for some 
$N>d$:
\eekv{ge.8}
{
&&\hbox{If }u\in B\hbox{ and }\widetilde{u}:\widetilde{\Gamma }\to {\bf C} 
\hbox{ satisfies } \vert \widetilde{u}(\widetilde{\gamma }) \vert\le 
\sum_{\gamma \in \Gamma }\langle \widetilde{\gamma }-\gamma  \rangle^{-N}
\vert u(\gamma ) \vert,\ \widetilde{\gamma }\in \widetilde{\Gamma },
}
{
&&\hbox{then }\widetilde{u}\in \widetilde{B}\hbox{ and }
\Vert \widetilde{u}\Vert_{\widetilde{B}}\le C\Vert u\Vert_B, \hbox{ where } 
C\hbox{ is 
independent of }u,\widetilde{u}.
}

\par If \no{ge.8} holds for one $N>d$ and $M>d$ then it also holds with $N$ 
replaced by $M$. This is obvious when $M\ge N$ and if $d<M<N$, it follows from the observation that 
$$
\langle \widetilde{\gamma }-\gamma  \rangle^{-M}\le C_{N,M} 
\sum_{\widetilde{\beta }\in \widetilde{\Gamma }}\langle \widetilde{\gamma }
-\widetilde{\beta} \rangle^{-M}\langle \widetilde{\beta }-\gamma  \rangle^{-N}
$$
(cf.~\no{co.19.5}, where $I$ is the integral in \no{co.18}, $2n$ is 
replaced by $d$, and we take $M_1=M_2=0$), which allows us to write
$$
\sum_{\gamma \in \Gamma }\langle \widetilde{\gamma }-\gamma  \rangle^{-M}
\vert u(\gamma ) \vert\le C_{N,M}\langle \cdot  \rangle^{-M}*v,
$$
where $v(\beta ):=\sum_{\gamma }\langle \widetilde{\beta }
-\gamma  \rangle^{-N}\vert u(\gamma ) \vert$ and $v$ belongs to 
$\widetilde{B}$ since \no{ge.8} holds.
\begin{dref}\label{ge1}\rm
Let $\Gamma ,\widetilde{\Gamma }$ be two lattices in $E$ and let 
$B,\widetilde{B}$ be Banach spaces of functions on $\Gamma $ and 
$\widetilde{\Gamma }$ respectively, satisfying \no{ge.1}, \no{ge.4}. 
Then we say that $B\equiv \widetilde{B}$, if $B\prec \widetilde{B}$ and 
$\widetilde{B}\prec B$. Notice that this is an equivalence relation.
\end{dref}

\par We can now introduce our generalized symbol spaces. With 
$E\simeq {\bf R}^d$ as above, let $\Gamma \subset E\times E^*$ be a lattice 
and $B\subset \ell^\infty $ a Banach space satisfying \no{ge.1}, \no{ge.4}. 
Let $a\in {\cal S}'(E)$.

\begin{dref}\label{ge2}\rm We say that $a\in \widetilde{S}(m,B)$ if the function
$$
\Gamma \ni \gamma \mapsto {1\over m(\gamma )}\Vert \chi_\gamma ^w a\Vert
$$
belongs to $B$. Here $\chi_\gamma $ is the partiction of unity \no{sy.2}.
\end{dref} 

Proposition \ref{sy2} extends to 
\begin{prop}\label{ge3}
$\widetilde{S}(m,B)$ is a Banach space with the natural norm. If we replace 
$\Gamma ,\chi,B$ by $\widetilde{\Gamma },\widetilde{\chi},\widetilde{B}$, 
having the same properties, and with $\widetilde{B}\subset \ell^\infty 
(\widetilde{\Gamma }) $ equivalent to $B$, and if we further replace 
the $L^2$ norm by the $L^p$
norm for any $p\in [1,\infty ]$, we get the same space, equipped with an 
equivalent norm.
\end{prop}

\begin{proof}
It suffices to follow the proof of Proposition \ref{sy2}: From the estimate 
\no{sy.3.5} we get for any $N\ge 0$,
$$
{1\over m(\widetilde{\gamma })}\Vert \chi_{\widetilde{\gamma }}^wa\Vert
_{L^p}\le C_{p,N}\sum_{\gamma \in \Gamma }\langle \widetilde{\gamma }-\gamma 
 \rangle^{-n}{1\over m(\gamma )}\Vert \chi_\gamma ^wa\Vert_{L^2},
$$
where we also used that $m$ is an order function. Hence, since $B$, 
$\widetilde{B}$ are equivalent,
$$
\Vert {1\over m(\cdot )}\Vert \widetilde{\chi}^w_a\cdot \Vert_{L^p}\Vert_{\widetilde{B}}\le 
\Vert {1\over m(\cdot )}\Vert {\chi}^w_\cdot a \Vert_{L^2}\Vert_{{B}}.
$$
The reverse estimate is obtained the same way.
\end{proof}

\par As a preparation for the use of Bargmann transforms, we next develop a 
``continuous'' version of $B$-spaces; a kind of amalgam spaces 
in the sense of \cite{GrSt, Fe2, FoSt}.
 Let $\Gamma $ be a lattice in a 
$d$-dimensional real vector space $E$ and let $B\subset \ell^\infty (\Gamma )$
 satisfy \no{ge.1}, \no{ge.4}. Let $0\le \chi\in C_0^\infty (E)$ satisfy 
$\sum_{\gamma \in \Gamma }\tau_\gamma \chi >0$.
\begin{dref}\label{ge6}\rm
We say that the locally bounded measurable function $u:E\to {\bf C}$ is 
of class $[B]$, if there exists $v\in B$ such that 
\ekv{ge.8.3}
{
\vert u(x) \vert \le \sum_{\gamma \in \Gamma }v(\gamma )\tau_\gamma \chi(x).
}
\end{dref}

\par The space of such functions is a Banach space that we shall denote by 
$[B]$, equipped with the norm 
\ekv{ge.8.5}
{
\Vert u\Vert_{[B]}=\inf \{\Vert v\Vert_B;\, \no{ge.8.3}\hbox{ holds }\}.
}
This space does not depend on the choice of $\chi$ and we may actually 
characterize it as the space of all locally bounded measurable functions $u$
on $E$ such that 
\ekv{ge.9}
{
\vert u(x) \vert \le \sum_{\gamma \in \Gamma }w(\gamma )
\langle x-\gamma  \rangle^{-N},\hbox{ for some  }w\in B,
}
where $N>d$ is any fixed number. Clearly \no{ge.8} implies \no{ge.9}. 
Conversely, if $u$ satisfies \no{ge.9} and $\chi$ is as in Definition 
\ref{ge6}, then 
$$
\langle x \rangle^{-N}\le C\sum_{\alpha \in \Gamma }\langle \alpha  \rangle
^{-N}\tau_\alpha \chi (x),
$$
so if \no{ge.9} holds, we have,
\begin{eqnarray*}
\vert u(x) \vert&\le& C\sum_{\gamma }w(\gamma )\sum_{\alpha }
\langle \alpha  \rangle^{-N}\chi (x-(\gamma +\alpha ))\\
&=& C\sum_{\beta }(\langle \cdot  \rangle^{-N}*w)(\beta )\chi (x-\beta ),
\end{eqnarray*}
and $\langle \cdot  \rangle^{-N}*w\in B$. 

\par Similarly, the definition does not 
change if we replace $B\subset\ell^\infty(\Gamma ) $ by an equivalent space 
$\widetilde{B}\subset\ell^\infty(\widetilde{\Gamma }) $.

\par Let $m_1,m_2,m_3$ be order functions on $E_1\times E_2$, $E_2\times E_3$,
 $E_1\times E_3$ respectively, where $E_j$ is a real vectorspace of 
dimension $d_j$.
Let $\Gamma _j\subset E_j$ be lattices and let
$$
B_1\subset\ell^\infty (\Gamma _1\times \Gamma _2),\ 
B_2\subset\ell^\infty (\Gamma _2\times \Gamma _3),\ 
B_3\subset\ell^\infty (\Gamma _1\times \Gamma _3)
$$ 
be Banach spaces satisfying \no{ge.1}, \no{ge.4}.  Introduce the 
\begin{ass}\label{ge7}\rm
If $k_j\in m_jB_j$, $j=1,2$, then 
$$
k_3(\alpha ,\beta ):=\sum_{\gamma \in \Gamma _2}k_1(\alpha ,\gamma )
k_2(\gamma ,\beta )
$$
converges absolutely for every $(\alpha ,\beta )\in \Gamma _1\times 
\Gamma _3$. Moreover, $k_3\in m_3B_3$ and
$$
\Vert k_3/m_3\Vert_{B_3}\le C\Vert k_1/m_1\Vert_{B_1}\Vert k_2/m_2\Vert_{B_2}
$$
where $C$ is independent of $k_1,k_2$.
\end{ass}

\par Again, it is an easy exercise to check that the assumption is invariant 
under changes of the lattices $\Gamma _j$ and the passage to corresponding 
equivalent $B$-spaces.
\begin{prop}\label{ge8}
We make the Assumption \ref{ge7}, where $B_j$ satisfy \no{ge.1}, \no{ge.4}. 
Let 
$K_j\in m_j[B_j]$ for $j=1,2$ in the sense that $K_j/m_j\in [B_j]$. 
Then the integral 
$$
K_3(x,y):=\int_{E_2}K_1(x,z)K_2(z,y)dz,\ (x,y)\in E_1\times E_3,
$$
converges absolutely and defines a function $K_3\in m_3[B_3]$. Moreover, 
$$
\Vert K_3/m_3\Vert_{[B_3]}\le C\Vert K_1/m_1\Vert_{[B_1]}
\Vert K_2/m_2\Vert_{[B_2]},
$$
where $C$ is independent of $K_1$, $K_2$.
\end{prop}
\begin{proof}
Write 
\begin{eqnarray*}
\vert K_1(x,z) \vert&\le& \sum_{\Gamma _1\times \Gamma _2}
k_1(\alpha ,\gamma )\chi^{(1)}(x-\alpha ,z-\gamma )\\
\vert K_2(z,y) \vert&\le& \sum_{\Gamma _2\times \Gamma _3}
k_2(\gamma ,\beta  )\chi^{(2)}(z-\gamma ,y-\beta ),
\end{eqnarray*}
with $\chi^{(1)}\in C_0^\infty (E_1\times E_2)$,
$\chi^{(2)}\in C_0^\infty (E_2\times E_3)$ as in Definition \ref{ge6} and 
with $k_j\in m_jB_j$. Then
\begin{eqnarray*}
\vert K_3(x,y)&\le&\int_{E_2}\vert K_1(x,z) \vert \vert K_2(z,y) \vert dz\\
&\le &
\sum_{(\alpha ,\beta )\in\Gamma _1\times \Gamma _3\atop \gamma ,\gamma '
\in \Gamma _2} k_1(\alpha ,\gamma )k_2(\gamma ',\beta )
F(x-\alpha ,y-\beta ;\gamma -\gamma '),
\end{eqnarray*}
where
\begin{eqnarray*}
F(x,y;\gamma -\gamma ')&=&\int \chi^{(1)}(x,z-\gamma )\chi^{(2)}(z-\gamma ',y)
dz\\
&=&\int \chi^{(1)}(x,z-(\gamma -\gamma '))\chi^{(2)}(z,y)dz.
\end{eqnarray*}
We notice that $0\le F(x,y;\gamma )\in C_0^\infty (E_1\times E_3)$ and that $F(x,y;\gamma )\not\equiv 0$ only for finitely many $\gamma \in \Gamma $. Hence 
for some $R_0>0$,
$$
\vert K_3(x,y) \vert\le \sum_{\vert \gamma  \vert\le R_0}\sum_{(\alpha ,
\beta )\in \Gamma _1\times \Gamma _3}\Big(\sum_{\gamma '}k_1(\alpha ,\gamma '+
\gamma )k_2(\gamma ',\beta )\Big) F(x-\alpha ,y-\beta ;\gamma ).
$$ 
Since
$$
{1\over m_1(\cdot ,\cdot \cdot )}k_1(\cdot ,\cdot\cdot  +\gamma )\in B_1,
$$
for every fixed $\gamma $, and
$
k_2/m_2\in B_2,
$
the assumption \ref{ge7} implies that 
$$
k_3(\alpha ,\beta ;\gamma ):=\sum_{\gamma '}k_1(\alpha ,\gamma '+\gamma )k_2(\gamma ',\beta )\in m_3B_3,
$$
for every $\gamma \in \Gamma $.

The proposition follows.
\end{proof}

We next generalize \no{di.1}. Let $F={\bf R}^d$ and define $T:L^2(F)\to 
H_\Phi (F^{\bf C})$ as in \no{l2.8}--\no{l2.11}. Let $m$ be an order function 
on $F\times F^*$, let $\Gamma \subset F\times F^*$ be a lattice and let 
$B\subset\ell^\infty (\Gamma )$ satisfy \no{ge.1}, \no{ge.4}. Then we get
\begin{prop}\label{ge9} we have
\ekv{ge.10}
{
\widetilde{S}(m,B)=\{ u\in {\cal S}'(F);\, {1\over m}\Big( (e^{-\Phi } 
Tu)\circ \pi \circ \kappa _T\Big)\in [B]\},
}
where $\pi :\Lambda _\Phi \ni (x,\xi )\mapsto x\in F^{\bf C}$ is the 
natural projection.
\end{prop}
\begin{proof}
This will be a simple extension of the proof of \no{di.1}. As there, we 
identify 
$F^{\bf C}$ with $F\times F^*$ by means of $\pi \circ \kappa _T$
and work on the latter space. Assume first that $u\in \widetilde{S}(m,B)$ 
and write 
$u=\sum_{\gamma \in \Gamma }\psi _\gamma ^w\chi_\gamma ^w u$ as in Lemma 
\ref{sy3}, so that $(\Vert \chi_\gamma ^w u\Vert )_{\gamma \in \Gamma }\in mB$.
Using \no{di.1.5}, we see that 
$$
| e^{-\Phi /h}Tu(x)|\le C_N \sum_{\gamma \in \Gamma } \Vert \chi_\gamma ^w u
\Vert \langle x-\gamma \rangle^{-N}, 
$$
and hence $e^{-\Phi }Tu\in m[B]$, i.e. $u$ belongs to the right hand side
of \no{ge.10} (with the identification $\pi \circ \kappa _T$).

\par Conversely, if $e^{-\Phi }Tu\in m[B]$, then since the effective kernel of 
$\chi_\gamma ^w$ satisfies \no{di.1.5}, we see that 
$$
|
e^{-\Phi }T\chi_\gamma ^w u(x) |\le C_N \int
 \langle x-\gamma \rangle^{-N} 
\langle y-\gamma  \rangle^{-N} 
\sum_{\alpha \in \Gamma }\langle y-\alpha \rangle ^{-N} a_\alpha dy,
$$
where $(a_\alpha )\in mB$. It follows that 
$$
|
e^{-\Phi }T\chi_\gamma ^w u(x) |\le \widetilde{C}_N 
 \langle x-\gamma \rangle^{-N} 
\sum_{\alpha \in \Gamma }\langle \gamma -\alpha \rangle ^{-N} a_\alpha 
=\widetilde{C}_N \langle x-\gamma  \rangle ^{-N}b_\gamma , 
$$
where $(b_\gamma )_{\gamma \in \Gamma }\in mB$, and hence 
$
\Vert \chi_\gamma ^w u\Vert\le \widehat{C}_N b_\gamma 
$,
so $u\in \widetilde{S}(m,B)$.
\end{proof}

\par
From this, we deduce as in \no{di.14} that if $a\in{\cal S}'(E)$, 
$E=F\times F^*$, then $a\in \widetilde{S}(m,B)$ iff
\ekv{ge.11}
{
K^{\rm eff\,}_{a^w}(t-{1\over 2}Jt^*,t+{1\over 2}Jt^*)\in m[B],
}
where $K^{\rm eff}_{a^w}$ is the effective kernel of $a^w$ in 
\no{di.4.5}, \no{di.5} after identification of ${\bf C}^d=F^{\bf C}$ with $E$
via the map $\pi \circ\kappa _T=E\to F^{\bf C}$. We recall the identity 
\no{di.15} for the composition of two symbols.

\par \no{ge.11} can also be written 
\ekv{ge.12}
{
K^{\rm eff}_{a^w}(x,y)\in \widetilde{m}[\widetilde{B}],
\hbox{ where } \widetilde{m}=m\circ q,\ [\widetilde{B}]=[B]\circ q,
}
where $q$ is given in \no{co.15.5}.
\par The following generalization of Theorem \ref{co2} now follows from
Proposition \ref{ge8}.
\begin{theo}\label{ge10}
For $j=1,2,3$, let $m_j$ be an order function $E\times E^*$, where 
$E={\bf R}^n\times 
({\bf R}^n)^*$, let $\Gamma _j\subset E\times E^*$ be a lattice and let 
$B_j\subset 
\ell^\infty (\Gamma _j)$ satisfy \no{ge.1}, \no{ge.4}. Let $\widetilde{m}_j
=m_j\circ q$, $\widetilde{\Gamma }_j=q\inv (\Gamma _j)$, 
$\ell^\infty  (\widetilde{\Gamma }_j)\supset \widetilde{B}_j=B_j\circ q$. Assuming 
(as we may without loss of generality) that $\widetilde{\Gamma }_j=\Gamma \times \Gamma$ where $\Gamma \subset E$ is a lattice, we make the 
Assumption \ref{ge7} for $\widetilde{m}_j\widetilde{B}_j$.

\par Then if $a_j\in \widetilde{S}(m_j,B_j)$, $j=1,2$, the composition $a_3=
a_1\# a_2$ is well defined and belongs to $\widetilde{S}(m_3,B_3)$, in the sense that the corresponding composition of effective kernels in \no{di.15} is given by an absolutely convergent integral and $K^{\rm eff}_{a_3^w}\in 
\widetilde{m}_3[\widetilde{B}_3]$. 
\end{theo}

We next consider the action of pseudodifferential operators on generalized 
symbol spaces. Our result will be essentially a special case of the preceding 
theorem. We start by ``contracting'' Assumption \ref{ge7} to the case when $E_3=0$. 
\par Let $m_1,m_2,m_3$ be order functions on $E_1\times E_2$, $E_2$,
 $E_1$ respectively. Let $\Gamma _j\subset E_j$, $j=1,2$ be lattices and let
$$
B_1\subset\ell^\infty (\Gamma _1\times \Gamma _2),\ 
B_2\subset\ell^\infty (\Gamma _2),\ 
B_3\subset\ell^\infty (\Gamma _1)
$$ 
be Banach spaces satisfying \no{ge.1}, \no{ge.4}.  Assumption \ref{ge7} becomes 
\begin{ass}\label{ge11}\rm
If $k_j\in m_jB_j$, $j=1,2$, then 
$$
k_3(\alpha )=\sum_{\beta \in \Gamma_2 }k_1(\alpha ,\beta )
k_2(\beta)
$$
converges absolutely for every $\alpha\in \Gamma _1$, and we have 
$k_3\in m_3B_3$. Moreover,
$$
\Vert k_3/m_3\Vert_{B_3}\le C\Vert k_1/m_1\Vert_{B_1}\Vert k_2/m_2\Vert_{B_2}
$$
where $C$ is independent of $k_1,k_2$.
\end{ass}

\par The corresponding ``contraction'' of Proposition \ref{ge8} becomes
\begin{prop}\label{ge12}
Let  Assumption \ref{ge11} hold, where $B_j$ satisfy \no{ge.1}, \no{ge.4}. Let 
$K_j\in m_j[B_j]$ for $j=1,2$. Then the integral 
$$
K_3(x):=\int_{E_2}K_1(x,z)K_2(z)dz,\ x\in E_1,
$$
converges absolutely and defines a function $K_3\in m_3[B_3]$. Moreover, 
$$
\Vert K_3/m_3\Vert_{[B_3]}\le C\Vert K_1/m_1\Vert_{[B_1]}
\Vert K_2/m_2\Vert_{[B_2]},
$$
where $C$ is independent of $K_1$, $K_2$.
\end{prop}

\par We get the following result for the action of pseudodifferential operators
on generalized symbol spaces.
\begin{theo}\label{ge13}
Let $m_2,m_3$ be order functions on $E={\bf R}^n\times ({\bf R}^n)^*$ and let 
$m_1$ be an order function on $E\times E^*$. Let $\widehat{\Gamma }\subset 
E\times E^*$ be a lattice such that $\widetilde{\Gamma }:=q\inv 
(\widehat{\Gamma })=\Gamma \times \Gamma $ where $\Gamma \subset E$ 
is a lattice. Let $\widehat{B}_1\subset \ell^\infty (\widehat{\Gamma })$, 
$B_2,B_3\subset\ell^\infty (\Gamma )$ 
satisfy \no{ge.1}, \no{ge.4}.  We make 
the Assumption \ref{ge11} with $\Gamma _1,\Gamma _2=\Gamma $ and with 
$m_1$, $B_1$ replaced with $\widetilde{m}_1=m_1\circ q$, 
$\widetilde{B}_1=B_1\circ q$, where $q$ is given in \no{co.15.5}.

\par Then, if $a_1\in \widetilde{S}(m_1,B_1)$, $u\in \widetilde{S}(m_2,B_2)$, 
the distribution $v=a_1^w(u)$ is well-defined in $\widetilde{S}(m_3,B_3)$ in the sense that 
$$
e^{-\Phi (x)}Tv(x)=\int K^{\rm eff}_{a_1^w}(x,y) e^{-\Phi (y)}Tu(y) L(dy),
$$
with $K^{\rm eff}_{a_1^w}(x,y)$ as in \no{di.4.5},
converges absolutely for every $x\in {\bf C}^n$ and 
$$
{1\over m_3}((e^{-\Phi }Tv)\circ \pi \circ \kappa _T)\in [B_3],
$$
as in \no{ge.10}.
\end{theo}

\par We shall finally generalize Theorem \ref{cp2}.
\begin{theo}\label{ge14}
Let $p\in [1,\infty ]$ and let $m$ be an order function on $E\times E^*$ where
$E={\bf R}^n\times ({\bf R}^n)^*$. Let $\Gamma \subset E$ be a lattice and 
$B\subset\ell^\infty (q(\Gamma \times \Gamma ))$  a Banach space 
satisfying \no{ge.1}, 
\no{ge.4}. Assume that 
\eekv{ge.13}
{&&
\hbox{if }(a_{\alpha ,\beta })_{\alpha ,\beta \in \Gamma }\in (m\circ q )B
\circ q ,
\hbox{ then }(a_{\alpha ,\beta })\in C_p(\ell^2(\Gamma ),\ell^2
(\Gamma ))
}
{&&\hbox{ and }\Vert (a_{\alpha ,\beta })\Vert_{C_p}
\le C\Vert (a_{\alpha ,\beta }) \Vert_{(m\circ q)B\circ q },
}
where $q $ is given in \no{co.15.5} and $C>0$ is independent of $(a_{\alpha ,\beta })$. Then there is a (new) constant $C>0$ such that 
\ekv{ge.14}
{
\hbox{If }a\in \widetilde{S}(m,B),\hbox{ then }a^w\in C_p(L^2,L^2)\hbox{ and }
\Vert a^w\Vert_{C_p}\le C\Vert a \Vert_{\widetilde{S}(m,B)}.
} 
\end{theo}

\par The proof of Proposition \ref{cp1} shows that the property \no{ge.13} 
is invariant under changes $(\Gamma ,B)\mapsto (\widetilde{\Gamma },
\widetilde{B})$ with $\widetilde{B}\subset\ell^\infty (q(\widetilde{\Gamma }
\times \widetilde{\Gamma }))$ 
equivalent to $B$.

\begin{proof}
We follow the proof of Theorem \ref{cp2}. Assume that \no{ge.13} holds and let
$a\in \widetilde{S}(m,B)$ be of norm $\le 1$. It suffices to show that 
$A_{{\rm eff}}:L^2({\bf C}^n)\to L^2({\bf C}^n)$ is in $C_p$ with norm $\le C$,
where $A_{\rm eff}$ is given in \no{cp.3.5} and $K^{\rm eff}$ there belongs 
to $m\circ q [B\circ q ]$, provided that we identify ${\bf C}^n$ with 
$E$ via $\pi \circ \kappa _T$.

\par We see that we still have 
\no{cp.8} where \no{cp.9} should be replaced by 
\eekv{ge.15}
{&&
\vert\nabla _x^k\nabla _y^\ell K_{\alpha ,\beta }(x,y)\vert \le C_{k,\ell}
a_{\alpha ,\beta },\ \vert x-\alpha  \vert,\, \vert y-\beta  \vert\le C_0,
}{&&
(a_{\alpha ,\beta })_{\alpha ,\beta \in \Gamma }\in (m\circ q )B\circ 
q ,\ \alpha ,\beta \in \Gamma .
}
Write $A_{\rm eff}=W^*{\cal A}W$ as in \no{cp.11}, 
$$
{\cal A}:\bigoplus_{\beta \in \Gamma }L^2(\Omega _\beta )\to 
\bigoplus_{\beta \in \Gamma }L^2(\Omega _\beta ),
\quad {\cal A}=(A_{\alpha ,\beta }). 
$$
The matrix elements $K_{\alpha ,j;\beta ;k}$ of $A_{\alpha ,\beta }$ now 
obey the estimate (cf.~\no{cp.12}):
\ekv{ge.16}
{
\vert K_{\alpha ,j;\beta ,k} \vert\le C_N\langle j \rangle^{-N}
\langle k \rangle^{-N} a_{\alpha ,\beta }
}
with $a_{\alpha ,\beta }$ as in \no{ge.16}. Using \no{ge.13}, this leads to \no{cp.12.4} and from that point on the proof is identical to that of Theorem 
\ref{ge14}.
\end{proof}

 \end{document}